\documentclass[11pt]{article}
\usepackage[utf8]{inputenc}
\usepackage{algorithm}
\usepackage{algpseudocode}
\usepackage{geometry}
\usepackage{secdot}
\usepackage{amsmath}
\usepackage{amsbsy,amsmath,amsthm,amssymb}
\usepackage{amsfonts}
\usepackage{mathrsfs,bm}
\usepackage{natbib}
\usepackage{graphicx}
\usepackage{graphics}
\usepackage{float}
\usepackage{hyperref}
\usepackage{multirow}
\usepackage{caption} 
\usepackage{bm}
\usepackage{xurl}

\hypersetup{
	colorlinks=true,
	linkcolor=blue,
	filecolor=blue,      
	urlcolor=blue,
	citecolor=blue,
}

\usepackage[english]{babel}
\usepackage{lastpage}
\usepackage{verbatim}
\usepackage{pdfpages}
\usepackage{authblk}

\newcommand{\df}{\mbox{d}}
\newcommand{\mix}{P_{\scriptsize\mbox{mix}}}
\newcommand{\mixdensity}{p_{\scriptsize\mbox{mix}}}
\newcommand{\betareal}{\bm{\beta}}
\newcommand{\sigmareal}{\bm{\varsigma}}

\newtheorem{theorem}{Theorem}
\newtheorem{lemma}[theorem]{Lemma}

\newtheorem{proposition}[theorem]{Proposition}
\newtheorem{definition}[theorem]{Definition}

\newtheorem{remark}[theorem]{Remark}

\title{\bf{Convergence Analysis of Data Augmentation Algorithms for Bayesian Robust Multivariate Linear Regression with Incomplete Data}}
\author{Haoxiang Li}
\author{Qian Qin}
\author{Galin L. Jones}

\affil{School of Statistics \\ University of Minnesota}

\date{}

\begin{document}

\maketitle

\begin{abstract}
Gaussian mixtures are commonly used for modeling heavy-tailed error distributions in robust linear regression.
Combining the likelihood of a multivariate robust linear regression model with a standard improper prior distribution yields an analytically intractable posterior distribution that can be sampled using a data augmentation algorithm.
When the response matrix has missing entries, there are unique challenges to the application and analysis of the convergence properties of the algorithm.
Conditions for geometric ergodicity are provided when the incomplete data have a ``monotone" structure.
In the absence of a monotone structure, an intermediate imputation step is necessary for implementing the algorithm.
In this case, we provide sufficient conditions for the algorithm to be Harris ergodic.
Finally, we show that, when there is a monotone structure and intermediate imputation is unnecessary, intermediate imputation slows the convergence of the underlying Monte Carlo Markov chain, while post hoc imputation does not.
An R package for the data augmentation algorithm is provided.
\end{abstract}

\section{Introduction}
Markov chain Monte Carlo (MCMC) algorithms are frequently used for sampling from intractable probability distributions in Bayesian statistics.
We study the convergence properties of a class of MCMC algorithms designed for Bayesian multivariate linear regression with heavy-tailed errors.
Contrary to most existing works in the area of MCMC convergence analysis, we focus on scenarios where the data set is incomplete. The existence of missing data substantially complicates the analysis.

Gaussian mixtures are suitable for constructing heavy-tailed error distributions in robust linear regression models \citep{zellner1976bayesian,lange1993normal,fernandez2000bayesian}.
In a Bayesian setting where a simple improper prior is used, the mixture representation facilitates a data augmentation (DA) MCMC algorithm \citep{liu1996bayesian} that can be used to sample from the posterior distribution of the regression coefficients and error scatter matrix.
When there are no missing data, this algorithm is geometrically ergodic under regularity conditions \citep{roy2010monte,hobert2018convergence}.
Loosely speaking, geometric ergodicity means that the Markov chain converges to the posterior distribution at a geometric, or exponential rate.
Geometric ergodicity is important because it guarantees the existence of a central limit theorem for ergodic averages, which is in turn crucial for assessing the accuracy of Monte Carlo estimators \citep[]{chan1994discussion,jones2001honest,jones2004markov,flegal2008markov, vats2018strong, vats2019multivariate, jones2021markov}.

The current work studies the algorithm when the response matrix has missing entries.
An incomplete data set brings unique challenges to the implementation and analysis of the DA algorithm.
For instance, establishing posterior propriety becomes much more difficult than when there are no missing values.
If the posterior measure is not proper, the DA algorithm will produce nonsensical results \citep{hobert1996effect}.
However, if one can show that the algorithm is geometrically ergodic, then the posterior distribution is guaranteed to be proper.

When the missing data have a certain ``monotone" structure, the DA algorithm can be carried out without an intermediate step to impute the missing data.
In this case, we establish geometric ergodicity under conditions on the error distribution and the amount of incomplete response components.
Roughly speaking, when the mixing distribution in the Gaussian mixture representation of the error distribution places little mass near the origin and the number of incomplete components is not too large, the DA algorithm is geometrically ergodic. The conditions are satisfied by many mixing distributions, including distributions with finite supports, log-normal, generalized inverse Gaussian, and inverse Gamma or Fréchet with shape parameter greater than $d/2$, where $d$ is the dimension of the response variable.  Some Gamma, Weibull, and F distributions also satisfy the conditions.
A post hoc imputation step can be added to fill in the missing values and the convergence properties of the DA algorithm will be unaffected. 

When the missing data do not posses a monotone structure, some missing entries need to be imputed to implement the DA algorithm.
This results in a data augmentation algorithm with an intermediate (as opposed to post hoc) imputation step, which we call a DAI algorithm for short.
We provide sufficient conditions for the DAI algorithm to be Harris ergodic. 
Harris ergodicity is weaker than geometric ergodicity, but it guarantees posterior propriety as well as the existence of a law of large numbers for ergodic averages \citep[Theorem 17.1.7]{meyn2005markov}.

When the missing data have a monotone structure, both the DA (with or without post hoc imputation) and DAI algorithms can be applied.
However, we show that the DA algorithm converges in $L^2$ at least as fast as the DAI algorithm.

Our key strategy is to draw a connection from cases where the data set is incomplete to the standard case where the data set is fully observed.
This allows for an analysis of the former using tools built for the latter.

Finally, we provide an R package \texttt{Bayesianrobust} that implements the DA and DAI algorithms. The R package is available from \url{https://github.com/haoxiangliumn/Bayesianrobust}. 
While the algorithms are documented in \cite{liu1996bayesian}, we do not know of previous software packages that implement them.

The rest of is organized as follows.  Section \ref{sec:model} recounts the Bayesian robust linear regression model with incomplete data. In Section \ref{sec:algorithms}, we describe the DA and DAI algorithms. 
Our main results are in Section \ref{sec:ca}, where we provide conditions for geometric ergodicity of the DA algorithm and Harris ergodicity of the DAI algorithm.
The section also contains a comparison between the DA and DAI algorithms in terms of $L^2$ convergence rate. 
Section \ref{sec:num} presents a numerical experiment. 

\section{Robust Linear Regression with Incomplete Data}\label{sec:model}

Let $(\bm{Y}_i,\bm{x}_i), \; i\in\{1,\dots,n\}$ be~$n$ independent data points, where 
$\bm{x}_i = (x_{i,1}, \dots, x_{i,p})^\top$ is a $p\times 1$ vector of known predictors, and $\bm{Y}_i = (Y_{i,1},\dots,Y_{i,d})^\top$ is a $d \times 1$ random vector. 
Consider the multivariate linear regression model
\begin{equation} \nonumber
	\bm{Y}_i = \bm{B}^\top \bm{x}_i + \bm{\Sigma}^{1/2} \bm{\varepsilon}_i,\ i \in \{1,\dots,n\},  
\end{equation}
where $\bm{B}$ is a $p\times d$ matrix of unknown regression coefficients, $\bm{\Sigma}$ is a $d\times d$ unknown positive definite scatter matrix, and $\bm{\varepsilon}_i, \; i\in\{1,\dots,n\}$, are $d\times1$ iid random errors.
Let $\bm{Y}$ be the $n\times d$ response matrix whose $i$th row  is $\bm{Y}_i^\top$,  and let $\bm{x}$ be the $n \times p $ design matrix whose $i$th row is $\bm{x}_i^\top$.

To allow for error distributions with potentially heavy tails,
assume that the distribution of each $\bm{\varepsilon}_i$ is described by a scale mixture of multivariate normal densities, which takes the form
\begin{equation} \nonumber
	f_{\text{err}}(\bm{\epsilon}) = \int_{0}^{\infty} \frac{w^{d/2}}{(2\pi)^{d/2}} \exp \left(-\frac{w}{2} \bm{\epsilon}^{T} \bm{\epsilon} \right) \mix(\df w) , \quad \bm{\epsilon} \in \mathbb{R}^d,  
\end{equation}
where $\mix(\cdot)$ is a probability measure on $(0,\infty)$ referred to as the mixing distribution. 
Gaussian mixtures constitute a variety of error distributions, and are widely used for robust regression.
For instance, when $\mix(\cdot)$ corresponds to the $\text{Gamma}(v/2,v/2)$ distribution for some $v > 0$, i.e., $\mix(\cdot)$ has a density with respect to the Lebesgue measure given by
\[
\mixdensity(w)  \propto w^{v/2-1}\exp(vw/2) , \quad w>0,
\]
the errors follow the multivariate~$t$ distribution with~$v$ degrees of freedom, which has density function
\[
f_{\text{err}}(\bm{\epsilon}) \propto (1+\bm{\epsilon}^\top\bm{\epsilon}/v)^{-(d+v)/2}, \quad \bm{\epsilon} \in \mathbb{R}^d.
\]
See Section~\ref{ssec:geo} for more examples.

Consider a Bayesian setting. Throughout, we use $(\betareal, \sigmareal)$ to denote realized values of $(\bm{B}, \bm{\Sigma})$. Assume that $(\bm{B}, \bm{\Sigma})$ has the following prior density:
\begin{equation}\label{prior}
	p_{\scriptsize\mbox{prior}}(\betareal, \sigmareal)  	\propto |\sigmareal|^{-(m+1)/2} \exp \left\{-\frac{1}{2} \operatorname{tr} \left(\sigmareal^{-1} \bm{a} \right)\right\}, \quad \betareal \in \mathbb{R}^{p \times d}, \; \sigmareal \in S_+^{d \times d},
\end{equation}
where $m\in \mathbb{R}$, $\bm{a} \in S_+^{d \times d}$, and $S_+^{d \times d}$ is the convex cone of $d \times d$ (symmetric) positive semi-definite real matrices. 
Equation \eqref{prior} defines a class of commonly used default priors \citep{liu1996bayesian,fernandez1999multivariate}. 
For instance, the independence Jeffrey's prior corresponds to $m=d$ and $\bm{a}=\bm{0}$.
Denote by $P_{\scriptsize\mbox{prior}}$ the measure associated with~$p_{\scriptsize\mbox{prior}}$. 
The model of interest can then be written in the hierarchical form:
\begin{equation} \nonumber
	\begin{aligned}
		\bm{Y}_i \mid \bm{W}, \bm{B}, \bm{\Sigma} &\stackrel{\scriptsize\mbox{ind}}{\sim} \mbox{N}_d \left(  \bm{B}^\top \bm{x}_i, W_i^{-1} \bm{\Sigma} \right), \; i \in\{1,\dots,n\}; \\
		W_i \mid \bm{B}, \bm{\Sigma} &\stackrel{\scriptsize\mbox{ind}}{\sim} \mix(\cdot),\; i \in\{1,\dots,n\}; \\
		\bm{B},\bm{\Sigma} &\sim P_{\scriptsize\mbox{prior}}(\cdot),
	\end{aligned}
\end{equation}
where $\mbox{N}_d$ denotes $d$-variate normal distributions. Let 
$\bm{W} := (W_1,\dots,W_n)^\top$ be a vector of latent random variables.
The joint measure of $(\bm{Y},\bm{W},\bm{B},\bm{\Sigma})$ is given by
\begin{equation} \nonumber
	\begin{aligned}
		P_{\scriptsize\mbox{joint}}(\df \bm{y}, \df \bm{w}, \df \betareal, \df \sigmareal) = p_{\scriptsize\mbox{joint}}(\bm{y}, \bm{w}, \betareal,  \sigmareal) \, \df \bm{y} \, \mix(\df \bm{w}) \, \df \betareal \, \df \sigmareal,
	\end{aligned}
\end{equation}
where, for $\bm{y} = (\bm{y}_1,\dots,\bm{y}_n)^\top \in \mathbb{R}^{n \times d}$,  $\bm{w} = (w_1,\dots,w_n)^\top \in (0,\infty)^n$, $\betareal \in \mathbb{R}^{p \times d}$, and $\sigmareal \in S_+^{d \times d}$,
\begin{equation}  \label{eq:fullmodel}
	p_{\scriptsize\mbox{joint}}(\bm{y}, \bm{w}, \betareal,  \sigmareal)  
	\propto  \, p_{\scriptsize\mbox{prior}}(\betareal, \sigmareal) \prod_{i=1}^n \frac{w_i^{d/2}}{ |\sigmareal|^{1/2}} \exp \left\{ -\frac{1}{2}w_i \left(\bm{y}_i - \bm{\beta}^\top \bm{x}_i \right)^\top \sigmareal^{-1} \left(\bm{y}_i - \bm{\beta}^\top \bm{x}_i \right) \right\}.
\end{equation}
Throughout, conditional distributions concerning $(\bm{Y},\bm{W},\bm{B},\bm{\Sigma})$ are to be uniquely defined through the density $p_{\scriptsize\mbox{joint}}$.


We assume that~$\bm{x}$ is fully observed, but~$\bm{Y}$ may contain missing values.
The missing structure can be described by an $n \times d$ random matrix $\bm{K} = (K_{i,j})_{i=1}^n{}_{j=1}^d$, with $K_{i,j} = 1$ indicating $Y_{i,j}$ is observed, and $K_{i,j} = 0$ indicating $Y_{i,j}$ is missing.
For a realized response matrix $\bm{y} = (y_{i,j})_{i=1}^n{}_{j=1}^d \in \mathbb{R}^{n \times d}$ and a realized missing structure $\bm{k} = (k_{i,j})_{i=1}^n{}_{j=1}^d \in \{0,1\}^{n \times d}$, let $\bm{y}_{(\bm{k})}$ be the array $(y_{i,j})_{(i,j) \in A(\bm{k}) }$, where $A(\bm{k}) = \{(i,j): k_{i,j} = 1 \}$.
Then $\bm{Y}_{(\bm{K})}$ gives the observable portion of~$\bm{Y}$.
In practice, instead of observing $\bm{Y}$, one sees instead $(\bm{K}, \bm{Y}_{(\bm{K})})$.
Throughout, we assume that the missing data mechanism is ignorable, 
which means the following.
\begin{definition}
	The missing data mechanism is ignorable if, for any $\bm{k} \in \{0,1\}^{n \times d}$ and almost every $\bm{y} \in \mathbb{R}^{n \times d}$, the posterior distribution of $(\bm{B},\bm{\Sigma})$ given $(\bm{K}, \bm{Y}_{(\bm{K})}) = (\bm{k}, \bm{y}_{(\bm{k})})$ is the same as the conditional distribution of $(\bm{B},\bm{\Sigma})$ given $\bm{Y}_{(\bm{k})} = \bm{y}_{(\bm{k})}$.
\end{definition}
\noindent If, given $(\bm{B}, \bm{\Sigma}) = (\betareal,\sigmareal)$, the distribution of~$\bm{K}$ does not depend on $(\betareal,\sigmareal)$, and the data are ``realized missing at random"  \citep{seaman2013meant}, which can be implied by the fact that $\bm{K}$ is independent of~$\bm{Y}$, then the missing data mechanism is ignorable.  
From here on, fix a realized missing structure~$\bm{k} = (k_{i,j})_{i=1}^n{}_{j=1}^d \in \{0,1\}^{n \times d}$, and only condition on $\bm{Y}_{(\bm{k})}$ when studying posterior distributions.
Without loss of generality, assume that each row of~$\bm{k}$ contains at least one nonzero element, i.e., for each~$i$, $\bm{Y}_i$ has at least one observed entry.

To write down the exact form of the posterior, we introduce some additional notation.
Suppose that given $i \in \{1,\dots,n\}$, $k_{i,j} = 1$ if and only if $j \in \{j_1,\dots,j_{d_i}\}$ for some $d_i \in \{1,\dots,d\}$, where $j_1 < \dots < j_{d_i}$.
Given~$i$, let $\bm{c}_{i,(\bm{k})}$ be the $d_i \times d$ matrix that satisfies the following: 
For $\ell = 1,\dots,d_i$, in the $\ell$th row of $\bm{c}_{i,(\bm{k})}$, all elements except the $j_\ell$th one are~0, while the $j_\ell$th one is~1.
For instance, if $d=4$, $d_i = 2$, $j_1 = 1$, $j_2 = 3$, then
\[
\bm{c}_{i,(\bm{k})} = \left( \begin{array}{cccc}
	1 & 0 & 0 & 0 \\
	0 & 0 & 1 & 0
\end{array}
\right).
\] 
Then $\bm{Y}_{i,(\bm{k})} := \bm{c}_{i,(\bm{k})} \bm{Y}_i$ is a vector consisting of the observed components of $\bm{Y}_i$ if $\bm{K} = \bm{k}$, and we can write $\bm{Y}_{(\bm{k})}$ as $(\bm{Y}_{i,(\bm{k})} )_{i=1}^n$.
%
%
%
%
%
For a realized value of~$\bm{Y}$, say $\bm{y} = (y_{i,j})_{i=1}^n{}_{j=1}^d \in \mathbb{R}^{n \times d}$, denote by $\bm{y}_i$ the $i$th row of~$\bm{y}$ transposed, let $\bm{y}_{i,(\bm{k})} = \bm{c}_{i,(\bm{k})} \bm{y}_i$, and let $\bm{y}_{(\bm{k})} = ( \bm{y}_{i,(\bm{k})} )_{i=1}^n$.
Based on~ Equation \eqref{eq:fullmodel}, the posterior density of $(\bm{B},\bm{\Sigma} )$ given $\bm{Y}_{(\bm{k})} = \bm{y}_{(\bm{k})} 
$ has the following form:
\begin{equation} \nonumber 
	\pi_{\bm{k}} \left(\betareal,\sigmareal \mid \bm{y}_{(\bm{k})} \right) \propto   p_{\scriptsize\mbox{prior}}(\betareal,\sigmareal) \prod_{i=1}^n \int_0^{\infty} \frac{w^{d_i/2}}{ \left|\bm{c}_{i,(\bm{k})} \sigmareal \bm{c}_{i,(\bm{k})}^\top \right|^{1/2}} \exp \left( - \frac{r_{i,(\bm{k})} w}{2} \right) \, \mix(\df w),
\end{equation}
where, for $i \in \{1,\dots,n\}$,
\begin{equation}\label{eq:r}
	r_{i,(\bm{k})} = \left(\bm{y}_{i, (\bm{k})} - \bm{c}_{i,(\bm{k})} \betareal^\top \bm{x}_i \right)^\top \left( \bm{c}_{i,(\bm{k})} \sigmareal \bm{c}_{i,(\bm{k})}^\top \right)^{-1}  \left(\bm{y}_{i, (\bm{k})} - \bm{c}_{i,(\bm{k})} \betareal^\top \bm{x}_i \right).
\end{equation}

\begin{remark}
	Because the prior distribution is not proper, $\pi_{\bm{k}}(\cdot \mid \bm{y}_{(\bm{k})})$ is not automatically a proper probability density.
	As a side product of our convergence analysis, we give sufficient conditions for posterior propriety in Section \ref{sec:ca}.
\end{remark}

The posterior density $\pi_{\bm{k}}(\cdot \mid \bm{y}_{(\bm{k})})$ is almost always intractable in the sense that it is hard to calculate its features such as expectation and quantiles.
\cite{liu1996bayesian} proposed a data augmentation (DA) algorithm, or two-component Gibbs sampler, that can be used to sample from this distribution.
This is an MCMC algorithm that simulates a Markov chain $(\bm{B}(t), \bm{\Sigma}(t))_{t=0}^{\infty}$ that is reversible with respect to the posterior.
In the next section, we describe the algorithm in detail.

\section{The DA and DAI Algorithms} \label{sec:algorithms}
\subsection{Missing Structures} \label{ssec:missing}

To adequately describe the DA and DAI algorithms, we need to introduce some concepts regarding missing structure matrices.

A realized missing structure $\bm{k} = (k_{i,j})_{i=1}^n{}_{j=1}^d \in \{0,1\}^{n \times d}$ is said to be monotone if the following conditions hold:
\begin{enumerate}
	
	\item [(i)] If $k_{i,j} = 1$ for some $i \in \{1,\dots,n\}$ and $j \in \{1,\dots,d\}$, then $k_{i',j'} = 1$ whenever $i' \leq i$ and $j' \geq j$.
	\item [(ii)] $k_{i,d} = 1$ for $i \in \{1,\dots,n\}$.
\end{enumerate}

Note that in practice, there are cases where 
the observed missing structure can be re-arranged to become monotone by permuting the rows and columns of the response matrix.
If there are no missing data, the corresponding missing structure is monotone.

\begin{table}[]
	\centering
	\caption{The observed response under a monotone structure}\label{monotone}
	\begin{tabular}{c|cccc|}
		\cline{2-5}
		\multicolumn{1}{c|}{\multirow{3}{*}{Pattern 1}} & $y_{1,1}$       & $y_{1,2}$       & $\cdots$                   & $y_{1,d}$       \\
		\multicolumn{1}{c|}{}                           & $\vdots$                    & $\vdots$                    & $\ddots$                    & $\vdots$                   \\
		\multicolumn{1}{c|}{}                           & $y_{n_1,1}$     & $y_{n_1,2}$     & $\cdots$                   & $y_{n_1,d}$     \\ \cline{2-2}
		\multirow{3}{*}{Pattern 2}             & \multicolumn{1}{l|}{} & $y_{n_1+1,2}$   & $\cdots$                    & $y_{n_1+1,d}$   \\
		& \multicolumn{1}{l|}{} & $\vdots$                   & $\ddots$                    & $\vdots$                    \\
		& \multicolumn{1}{l|}{} & $y_{n_1+n_2,2}$ & $\cdots$                   & $y_{n_1+n_2,d}$ \\ \cline{3-3}
		$\vdots$                                             &                       & \multicolumn{1}{l|}{} & $\ddots$                    & $\vdots$                   \\ \cline{4-4}
		\multirow{3}{*}{Pattern $d$}             &                       &                       & \multicolumn{1}{l|}{} & $y_{n-n_d+1,d}$   \\
		&                       &                       & \multicolumn{1}{l|}{} & $\vdots$  \\
		&                       &                       & \multicolumn{1}{l|}{} & $y_{n,d}$       \\ \cline{2-5} 
	\end{tabular}
\end{table}

Let $\bm{k} = (k_{i,j})_{i=1}^n{}_{j=1}^d$ be monotone.
Then, for a realized response matrix $\bm{y} = (y_{i,j})_{i=1}^n{}_{j=1}^d \in \mathbb{R}^{n \times d}$, the elements in $\bm{y}_{(\bm{k})}$ can be arranged as in Table~\ref{monotone}.
We say that the $i$th observation belongs to pattern~$\ell$ for some $\ell \in \{1,\dots,d\}$ if $k_{i,j} = 1$ for $j \geq \ell$ and $k_{i,j} = 0$ for $j < \ell$.
There are~$d$ possible patterns. 
For $\ell \in \{1,\dots,d\}$, denote by $n_\ell(\bm{k})$  the number of observations that belong to pattern~$\ell$.
Let $\bm{y}_{(\bm{k},\ell)}$ be the $[\sum_{j=1}^\ell n_j(\bm{k})] \times (d-\ell+1)$ matrix whose $i$th row is
$( y_{i,\ell}, \dots, y_{i,d} ),$
and let $\bm{x}_{(\bm{k},\ell)}$ be the submatrix of~$\bm{x}$ formed by the first $\sum_{j=1}^\ell n_j(\bm{k})$ rows of~$\bm{x}$.
Denote by $\bm{y}_{(\bm{k},\ell)}:\bm{x}_{(\bm{k},\ell)}$ the matrix formed by attaching $\bm{x}_{(\bm{k},\ell)}$ to the right of $\bm{y}_{(\bm{k},\ell)}$.
We say that Condition \eqref{da-condition} holds for the pair $(\bm{k},\bm{y}_{(\bm{k})})$ if
\begin{equation} \label{da-condition} \tag{H1}
	r \left(\bm{y}_{(\bm{k},\ell)}:\bm{x}_{(\bm{k},\ell)} \right) = p + d - \ell + 1, \quad
	\sum_{j=1}^\ell n_j(\bm{k}) > p + d - m + \ell - 1 \; \text{ for } \ell \in \{1,\dots,d\}.
\end{equation}
Here, $r(\cdot)$ returns the rank of a matrix.
This condition is crucial for the DA algorithm to be well-defined and implementable.

\subsection{The DA algorithm} \label{ssec:da}

Fix a realized response $\bm{y} \in \mathbb{R}^{n \times d}$ and a realized missing structure $\bm{k} \in \{0,1\}^{n \times d}$.
Given the current state $(\bm{B}(t),\bm{\Sigma}(t)) = (\betareal,\sigmareal)$, the DA algorithm for sampling from $\pi_{\bm{k}}(\cdot \mid \bm{y}_{(\bm{k})})$ draws the next state $(\bm{B}(t+1),\bm{\Sigma}(t+1))$ using the following steps.
\begin{enumerate}
	\item \textbf{I step.} Draw $\bm{W}^* = (W_1^*,\dots,W_n^*)$ from the conditional distribution of $\bm{W}$ given $(\bm{B}, \bm{\Sigma}, \bm{Y}_{(\bm{k})}) = (\betareal, \sigmareal, \bm{y}_{(\bm{k})})$.
	Call the sampled value $\bm{w}$.
	\item \textbf{P step.} Draw $(\bm{B}(t+1), \bm{\Sigma}(t+1))$ from the conditional distribution of $(\bm{B}, \bm{\Sigma})$ given $(\bm{W}, \bm{Y}_{(\bm{k})}) = (\bm{w}, \bm{y}_{(\bm{k})})$.
\end{enumerate}
This algorithm simulates a Markov chain $(\bm{B}(t), \bm{\Sigma(t)})_{t=0}^{\infty}$ that is reversible with respect to $\pi_{\bm{k}}(\cdot \mid \bm{y}_{(\bm{k})})$.

For $i \in \{1, \dots, n\}$, let $d_i$ be the number of nonzero entries in the $i$th row of~$\bm{k}$.
(Note: if $\bm{k}$ is monotone, and the $i$th observation belongs to some pattern~$\ell$, then $d_i = d-\ell+1$.)
One can derive that the conditional distribution of $\bm{W} = (W_1,\dots,W_n)$ given $(\bm{B}, \bm{\Sigma}, \bm{Y}_{(\bm{k})}) = (\betareal, \sigmareal, \bm{y}_{(\bm{k})})$ is
\begin{equation} \label{eq:wconditional}
	P_{{\bm{k}}}(\df \bm{w} \mid \betareal,\sigmareal, \bm{y}_{(\bm{k})}) \propto \prod_{i=1}^n  w_i^{d_i/2} \exp \left( - \frac{r_{i,(\bm{k})} w_i}{2} \right) \, \mix( \df w_i),
\end{equation}
where $r_{i,(\bm{k})}$ is defined in Equation \eqref{eq:r} for $i \in \{1,\dots,n\}$.
To ensure that this conditional density is always proper, we assume throughout that
\begin{equation}\label{hh}\tag{H2}
	\int_{0}^{\infty} w^{d/2} \, \mix(\df w) <\infty.
\end{equation}
The conditional distribution $P_{_{\bm{k}}}(\cdot \mid \betareal, \sigmareal, \bm{y}_{(\bm{k})})$ corresponds to~$n$ independent univariate random variables.
For most commonly used mixing distributions $\mix(\cdot)$, this is not difficult to sample.
For instance, if $\mix(\cdot)$ is a Gamma distribution, $P_{_{\bm{k}}}(\cdot \mid \betareal, \sigmareal, \bm{y}_{(\bm{k})})$ is the product of~$n$ Gamma distributions.

The conditional distribution of $(\bm{B},\bm{\Sigma})$ given $(\bm{W}, \bm{Y}_{(\bm{k})}) = (\bm{w}, \bm{y}_{(\bm{k})})$ is not always tractable.
In fact, since an improper prior is used, this conditional is possibly improper.
\cite{liu1996bayesian} provided a method for sampling from this distribution when~$\bm{k}$ is monotone.
When $\bm{k}$ is monotone, and Condition \eqref{da-condition} holds from $(\bm{y},\bm{k})$, the conditional of $(\bm{B},\bm{\Sigma})$ given $(\bm{W}, \bm{Y}_{(\bm{k})}) = (\bm{w}, \bm{y}_{(\bm{k})})$ is proper for any $\bm{w} = (w_1,\dots,w_n) \in (0,\infty)^n$.
This conditional distribution can be sampled using chi-square and normal distributions. 
The method is intricate, so we relegate the details to Appendix \ref{app:dadetails}.

Let $\bm{k}_0 \in \{0,1\}^{n \times d}$ be the missing structure that corresponds to a completely observable response, i.e., all elements of~$\bm{k}_0$ are~1.
Denote by $\bm{Y}_{(\bm{k}_0-\bm{k})}$ the missing parts in the response. Sometimes, we are interested in the posterior distribution of $\bm{Y}_{(\bm{k}_0-\bm{k})}$ given $\bm{Y}_{(\bm{k})} = \bm{y}_{(\bm{k})}$, which takes the following form:
$$
P_{\bm{Y}_{(\bm{k}_0-\bm{k})}} \left(d\bm{z} \mid \bm{k}, \bm{y}_{(\bm{k})} \right) \propto   d\bm{z} \int_{S_{+}^{d\times d}} \int_{\mathbb{R}^{p\times d}} \int_{(0,\infty)^n} p_{\text{joint}}(\bm{y}^{\bm{z}}, \bm{w}, \betareal, \sigmareal) \mix(d\bm{w}) d\betareal d\sigmareal, 
$$
where $p_{\text{joint}}$ is given in Equation \eqref{eq:fullmodel}, and 
$\bm{y}^{\bm{z}}$ is a realized value of~$\bm{Y}$ such that $\bm{y}^{\bm{z}}_{(\bm{k})} = \bm{y}_{(\bm{k})}$ and $\bm{y}^{\bm{z}}_{(\bm{k}_0 - \bm{k})} = \bm{z}$.
To sample from $P_{\bm{Y}_{(\bm{k}_0-\bm{k})}} \left(\cdot \mid \bm{k}, \bm{y}_{(\bm{k})} \right)$, one can add a post hoc imputation step at the end of each DA iteration.

\begin{enumerate}
\setcounter{enumi}{2}
\item \textbf{Post hoc imputation step.} Draw $\bm{Z}(t+1)$ from 
the conditional distribution of $\bm{Y}_{(\bm{k}_0 - \bm{k})}$ given $(\bm{Y}_{(\bm{k})}, \bm{W}, \bm{B}, \bm{\Sigma}) = (\bm{y}_{(\bm{k})}, \bm{w}, \betareal^*,\sigmareal^*)$, where $(\betareal^*,\sigmareal^*)$ is the sampled value of $(\bm{B}(t+1), \bm{\Sigma}(t+1))$.
\end{enumerate}
Recall that, given $(\bm{W}, \bm{B}, \bm{\Sigma}) = (\bm{w}, \betareal^*,\sigmareal^*)$, $\bm{Y}$ has a multivariate normal distribution.
Then the conditional distribution of $\bm{Y}_{(\bm{k}_0 - \bm{k})}$ given $(\bm{Y}_{(\bm{k})}, \bm{W}, \bm{B}, \bm{\Sigma}) = (\bm{y}_{(\bm{k})}, \bm{w}, \betareal^*,\sigmareal^*)$ is also (univariate or multivariate) normal.
One can show that $(\bm{B}(t), \bm{\Sigma}(t), \bm{Z}(t))_{t=0}^{\infty}$ is a Markov chain whose stationary distribution is the posterior distribution of $(\bm{B}, \bm{\Sigma}, \bm{Y}_{(\bm{k}_0 - \bm{k})})$ given $\bm{Y}_{(\bm{k})} = \bm{y}_{(\bm{k})}$. 

We call the imputation step post hoc because it can be implemented at the end of the whole simulation, as long as the value of $\bm{w}$ is recorded in the I step of each iteration.
Post hoc imputation is possible because the I and P steps do not rely on the value of $\bm{Z}(t)$.
Naturally, post hoc imputation does not affect the convergence properties of the Markov chain.
Indeed, standard arguments \citep{roberts2001markov} show that $(\bm{B}(t), \bm{\Sigma}(t), \bm{Z}(t))_{t=0}^{\infty}$ is geometrically ergodic if and only if $(\bm{B}(t), \bm{\Sigma}(t))_{t=0}^{\infty}$ is geometrically ergodic.
Moreover, the $L^2$ convergence rates of the two chains, as defined in Section~\ref{ssec:convergence}, are the same.
Thus, when studying the convergence properties of the DA algorithm, we can restrict our attention to $(\bm{B}(t), \bm{\Sigma}(t))_{t=0}^{\infty}$ instead of  $(\bm{B}(t), \bm{\Sigma}(t), \bm{Z}(t))_{t=0}^{\infty}$ even if there is post hoc imputation.

\subsection{The DAI algorithm} \label{ssec:dai}

For two realized missing structures $\bm{k} = (k_{i,j}) \in \{0,1\}^{n \times d}$ and $\bm{k}' = (k_{i,j}') \in \{0,1\}^{n \times d}$, write $\bm{k} \prec \bm{k}'$ if (i) $\bm{k} \neq \bm{k}'$ and (ii) $\bm{k}'_{i,j} = 1$ whenever $\bm{k}_{i,j} = 1$.
That is, $\bm{k} \prec \bm{k}'$ if the observed response entries under structure~$\bm{k}$ is a proper subset of those under~$\bm{k}'$.
If $\bm{k} \prec \bm{k}'$, then $\bm{k}' - \bm{k}$ is a missing structure matrix that gives the entries that are missing under~$\bm{k}$, but not under~$\bm{k}'$.
In other words, when $\bm{k} \prec \bm{k}'$, $\bm{Y}_{(\bm{k}'-\bm{k})}$ is observed under~$\bm{k}'$, but not under~$\bm{k}$.

As usual, fix a realized response $\bm{y} \in \mathbb{R}^{n \times d}$ and missing structure~$\bm{k}$.
One may imagine that~$\bm{k}$ is not monotone, and the DA algorithm cannot be implemented efficiently.
The DAI algorithm, also proposed in \cite{liu1996bayesian}, overcomes non-monotinicity by imputing the value of $\bm{Y}_{(\bm{k}'-\bm{k})}$ in the I step, where $\bm{k}'$ is a monotone realized missing structure such that $\bm{k} \prec \bm{k}'$ chosen by the user.
We now describe the details of this algorithm.

The DAI algorithm associated with the triplet $(\bm{y}, \bm{k},\bm{k}')$ simulates a Markov chain $(\bm{B}(t),\bm{\Sigma}(t))_{t=0}^{\infty}$ that is reversible with respect to $\pi_{\bm{k}}(\cdot \mid \bm{y}_{(\bm{k})})$.
Given the current state $(\bm{B}(t),\bm{\Sigma}(t)) = (\betareal,\sigmareal)$, the DAI algorithm draws the next state $(\bm{B}(t+1),\bm{\Sigma}(t+1))$ using the following steps.
\begin{enumerate}
\item \textbf{I1 step.} Draw $\bm{W}^* = (W_1^*,\dots,W_n^*)$ from the conditional distribution of $\bm{W}$ given $(\bm{B}, \bm{\Sigma}, \bm{Y}_{(\bm{k})}) = (\betareal,\sigmareal, \bm{y}_{(\bm{k})})$, whose exact form is given in~ Equation \eqref{eq:wconditional}.
Call the sampled value $\bm{w}$.

\item \textbf{I2 step.} Draw $\bm{Z}(t+1)$ from 
the conditional distribution of $\bm{Y}_{(\bm{k}' - \bm{k})}$ given\\ $(\bm{Y}_{(\bm{k})}, \bm{W}, \bm{B}, \bm{\Sigma}) = (\bm{y}_{(\bm{k})}, \bm{w}, \betareal, \sigmareal)$.
Call the sampled value $\bm{z}$.

\item \textbf{P step.} Draw $(\bm{B}(t+1), \bm{\Sigma}(t+1))$ from the conditional distribution of $(\bm{B}, \bm{\Sigma})$ given $(\bm{w}, \bm{Y}_{(\bm{k})}, \bm{Y}_{(\bm{k}' - \bm{k})}) = (\bm{w}, \bm{y}_{(\bm{k})}, \bm{z})$.
\end{enumerate}

Recall that, given $(\bm{W}, \bm{B}, \bm{\Sigma}) = (\bm{w}, \betareal, \sigmareal)$, $\bm{Y}$ has a multivariate normal distribution.
Then in the I2 step, the conditional distribution of $\bm{Y}_{(\bm{k}' - \bm{k})}$ given $(\bm{Y}_{(\bm{k})}, \bm{W}, \bm{B}, \bm{\Sigma}) = (\bm{y}_{(\bm{k})}, \bm{w}, \betareal, \sigmareal)$ is (univariate or multivariate) normal.
In the P step, one is in fact sampling from the conditional distribution of $(\bm{B}, \bm{\Sigma})$ given $(\bm{W}, \bm{Y}_{(\bm{k}')}) = (\bm{w}, \bm{y}_{(\bm{k}')}^*)$, where $\bm{y}^*$ is a realized value of~$\bm{Y}$ such that $\bm{y}^*_{(\bm{k})} = \bm{y}_{(\bm{k})}$ and $\bm{y}^*_{(\bm{k}' - \bm{k})} = \bm{z}$.
To ensure that this step can be implemented efficiently, we need two conditions: (i) $\bm{k}'$ is monotone as we have assumed, and (ii) Condition \eqref{da-condition} holds for $(\bm{k}',\bm{y}_{(\bm{k}')}^*)$.
If these two conditions hold, then one can use methods in Appendix \ref{app:dadetails} to implement the P step.
Of course, in each step of the DAI algorithm, $\bm{z} = \bm{y}_{(\bm{k}'-\bm{k})}^*$ changes.
Despite this, due to the following result, which is easy to verify, we do not have to check (ii) in every step.

\begin{proposition} \label{pro:dai-welldefined}
Let $\bm{k}'$ be a monotone missing structure such that $\bm{k} \prec \bm{k}'$, and let $\bm{y}^* \in \mathbb{R}^{n \times d}$ be such that $\bm{y}_{(\bm{k})}^* = \bm{y}_{(\bm{k})}$.
Suppose that there exists a monotone structure $\bm{k}''$ such that $\bm{k}'' \prec \bm{k}$, and that Condition \eqref{da-condition} holds for $(\bm{k}'',\bm{y}_{(\bm{k}'')})$.
Then Condition \eqref{da-condition} holds for $(\bm{k}',\bm{y}_{(\bm{k}')}^*)$ regardless of the value of $\bm{y}_{(\bm{k}'-\bm{k})}^*$.
\end{proposition}

The DAI algorithm can be used to impute the missing values of $\bm{Y}$.
Indeed, $(\bm{B}(t), \bm{\Sigma}(t), \bm{Z}(t))_{t=0}^{\infty}$ is a Markov chain whose stationary distribution is the posterior distribution of $(\bm{B}, \bm{\Sigma}, \bm{Y}_{(\bm{k}' - \bm{k})})$ given $\bm{Y}_{(\bm{k})} = \bm{y}_{(\bm{k})}$. 
This is similar to the DA algorithm with post hoc imputation described in Section~\ref{ssec:da}, especially if all the entries of $\bm{k}'$ are 1.
Compared to the DA algorithm with post hoc imputation, the DAI algorithm can be implemented even when~$\bm{k}$ is not monotone.
However, the intermediate imputation step I2 cannot be performed after the whole simulation process since the P step of DAI relies on the value of $\bm{Z}(t+1)$.

Standard arguments show that $(\bm{B}(t), \bm{\Sigma}(t), \bm{Z}(t))_{t=0}^{\infty}$ and $(\bm{B}(t), \bm{\Sigma}(t))_{t=0}^{\infty}$ have the same convergence rate in terms of total variation and $L^2$ distances \citep{roberts2001markov,liu1994covariance}.
Thus, when studying the convergence properties of the DAI algorithm, we can restrict our attention to $(\bm{B}(t), \bm{\Sigma}(t))_{t=0}^{\infty}$ instead of  $(\bm{B}(t), \bm{\Sigma}(t), \bm{Z}(t))_{t=0}^{\infty}$ even if we care about imputing missing data.

\section{Convergence Analysis}\label{sec:ca}
\subsection{Preliminaries}\label{ssec:convergence}

We start by reviewing some general concepts regarding the convergence properties of Markov chains.
Let $(\mathcal{X}, \mathcal{F})$ be a measurable space.
Consider a Markov chain $(X(t))_{t=0}^{\infty}$ whose state space is $(\mathcal{X}, \mathcal{F})$, and let $K: \mathcal{X} \times \mathcal{F} \to [0,1]$ be its transition kernel.
For a signed measure~$\mu$ and a measurable function~$f$ on $(\mathcal{X},\mathcal{F})$,
let
\[
\mu f = \int_{\mathcal{X}} f(x) \, \mu(\df x),
\] 
and 
$K$ can act on~$\mu$ and~$f$ as follows: $$\mu K(A) = \int_{\mathcal{X}} K(x,A) \, \mu(dx), \; A\in \mathcal{F}, \quad Kf(x) = \int_{\mathcal{X}} f(y)K(x, dy), \; x\in\mathcal{X},$$
assuming that the integrals are well-defined.
Suppose that the chain has a stationary distribution~$\pi$, i.e., $\pi K(\cdot) = \pi(\cdot)$.
The chain is said to be reversible if, for $A_1,A_2 \in \mathcal{F}$,
\[
\int_{A_1 \times A_2} \pi(\df x) K(x, \df y) = \int_{A_2 \times A_1} \pi(\df x) K(x, \df y).
\]


For two probability measures $\mu_1$ and $\mu_2$ on $(\mathcal{X}, \mathcal{F})$, denote $||\mu_1(\cdot) - \mu_2(\cdot)||_{TV}$ as the  total variation (TV) distance between the two probability measures. 
For $t \in \mathbb{Z}_+$, let $K^t: \mathcal{X} \times \mathcal{F} \to [0,1]$ be the $t$-step  transition kernel of the chain, so that $K^1 = K$, and $\mu K^{t+1} (\cdot ) = (\mu K^t ) K(\cdot)$ for any probability measure~$\mu$ on $(\mathcal{X},\mathcal{F})$.
A $\phi$-irreducible aperiodic Markov chain with stationary distribution $\pi$ is \textit{Harris ergodic} if and only if 
$\lim_{t\to \infty }|| K^t(x,\cdot) - \pi(\cdot)||_{TV} = 0$, for all $x\in\mathcal{X}$ \citep{meyn2005markov,roberts2006harris}. 
The chain is said to be \textit{geometrically ergodic} if the chain is Harris ergodic and
\begin{equation}\label{ergodic}
|| K^t(x,\cdot) - \pi(\cdot)||_{TV}\leq M(x) \rho^t, \quad x \in \mathcal{X}, \; t\in \mathbb{Z}^+,
\end{equation}
for some $\rho\in[0,1)$ and $M: \mathcal{X} \to [0,\infty)$.
As mentioned in the Introduction, Harris ergodicity guarantees a law of large numbers for ergodic averages \citep[Theorem 17.1.7]{meyn2005markov}, and geometric ergodicity guarantees a central limit theorem for ergodic averages \citep[]{jones2001honest,jones2004markov,flegal2008markov}.

Another commonly used distance between probability measures is the $L^2$ distance.
Let $L^2(\pi)$ be the set of measurable functions $f: \mathcal{X} \to \mathbb{R}$ such that
\[
\pi f^2 := \int_{\mathcal{X}} f(x)^2 \, \pi(\df x) < \infty.
\]
The $L^2$ distance between two probability measures $\mu_1$ and $\mu_2$ on $(\mathcal{X}, \mathcal{F})$   is 
\[
\|\mu_1(\cdot) - \mu_2(\cdot) \|_2 = \sup_{f \in L^2(\pi)} |\mu_1 f - \mu_2 f|.
\]
Denote by $L_*^2(\pi)$ the set of probability measures~$\mu$ on $(\mathcal{X},\mathcal{F})$ such that~$\mu$ is absolutely continuous to~$\pi$ and that $\df \mu / \df \pi$, the density of~$\mu$ with respect to~$\pi$, is in $L^2(\pi)$.
We say the chain is $L^2$ geometrically ergodic if
\begin{equation} \label{ine:l2geo}
\| \mu K^t (\cdot) - \pi(\cdot) \|_2 \leq C(\mu) \rho^t, \quad \mu \in L_*^2(\pi), \; t\in \mathbb{Z}^+,
\end{equation}
for some $\rho \in [0,1)$ and $C: L_*^2(\pi) \to [0,\infty)$.  In some cases, e.g., when the state space is finite, the~$\rho$s in~ Equations \eqref{ergodic} and~\eqref{ine:l2geo} are exchangeable. For reversible Markov chains, geometric ergodicity implies $L^2$ geometric ergodicity \citep[Theorem 2.1 and Remark 2.3]{roberts1997geometric}.

The infimum of $\rho \in [0,1]$ such that Equation \eqref{ine:l2geo} holds for some $C: L_*^2(\pi) \to [0,\infty)$ is called the $L^2$ convergence rate of the chain.
The smaller this rate is, the faster the convergence.

\subsection{Geometric ergodicity of the DA algorithm}\label{ssec:geo}

To state our result, we define three classes of mixing distributions based on their behaviors near the origin.
These classes were first examined by \cite{hobert2018convergence} who analyzed the DA algorithm when the response matrix is fully observed.
We say that the mixing distribution $\mix(\cdot)$ is \textit{zero near the origin} if there exists $ \theta>0$ such that$\int_{0}^{\theta} \mix(\df w) = 0$.
Assume now that $\mix(\cdot)$ admits a density function $\mixdensity: (0,\infty) \to [0,\infty)$ with respect to the Lebesgue measure.
If there exists $ c>-1$, such that $\lim_{w\to 0 } {\mixdensity(w)}/{w^c}<\infty$, we say that $\mix(\cdot)$ is \textit{polynomial near the origin} with power~$c$. The mixing distribution $\mix({\cdot})$ is \textit{faster than polynomial near the origin} if, for all $ c>0$, there exists $ \kappa_{c}>0$ such that the ratio ${\mixdensity(w)}/{w^c}$ is strictly increasing on $(0,\kappa_{c})$.

Most commonly used mixing distributions fall into one of these three classes. 
Examples will be given after we state the main result of this section.

Again, fix an observed missing structure $\bm{k}$ and observed response $\bm{y}_{(\bm{k})}$.
Let $d_i$ be the number of nonzero elements in the $i$th row of~$\bm{k}$.
Recall that~$n$ is the number of observations,~$p$ is the number of predictors,~$d$ is the dimension of the responses, and~$m$ is a parameter in the prior distribution Equation \eqref{prior}.

\begin{theorem}\label{thm:da}
Consider the DA algorithm targeting $\pi_{\bm{k}}(\cdot \mid \bm{y}_{(\bm{k})})$, as described in Section~\ref{ssec:da}.
Suppose that Condition \eqref{hh} holds, and that the conditional distribution of $(\bm{B},\bm{\Sigma})$ given $(\bm{W}, \bm{Y}_{(\bm{k})}) = (\bm{w}, \bm{y}_{(\bm{k})})$ is proper for every $\bm{w} = (w_1,\dots,w_n)^\top \in (0,\infty)^n$.
If any one of the following conditions holds, then the posterior $\pi_{\bm{k}}(\cdot \mid \bm{y}_{(\bm{k})})$ is proper and the underlying Markov chain is geometrically ergodic. 
\begin{enumerate}\label{ta1}
	\item $\mix(\cdot)$ is zero near the origin;
	\item $\mix(\cdot)$ is faster than polynomial near the origin; or
	\item $\mix(\cdot)$ is polynomial near the origin with power $c> c_1$, where
	\[
	c_1=\frac{n-p+m-\min\{d_1,\dots,d_n\}}{2}.
	\]
\end{enumerate}
\end{theorem}

\begin{remark}
Recall that when $\bm{k}$ is monotone, the conditional distribution of $(\bm{B},\bm{\Sigma})$ given $(\bm{W}, \bm{Y}_{(\bm{k})}) = (\bm{w}, \bm{y}_{(\bm{k})})$ is proper for every $\bm{w} = (w_1,\dots,w_n)^\top \in (0,\infty)^n$ if Condition \eqref{da-condition} holds for $(\bm{k}, \bm{y}_{(\bm{k})})$.
\end{remark}

\begin{remark}
When $\mix(\cdot)$ has a density function with respect to the Lebesgue measure and $\bm{Y}$ is fully observed, Theorem~\ref{thm:da} reduces to the main result in \cite{hobert2018convergence}.
\end{remark}

The proof of Theorem \ref{thm:da} is in Appendix \ref{app:dageo}.
In what follows, we list commonly used mixing distributions that fall into the three categories in Theorem~\ref{thm:da}.
We also check whether each mixing distribution satisfies Condition \eqref{hh}.

\noindent\textbf{Zero near the origin}

When the mixing distribution $\mix(\cdot)$ is discrete with finite support, the mixing distribution is zero near the origin. This is the case when errors follow finite mixtures of Gaussian. 
Obviously, Condition \eqref{hh} holds in this case.

The Pareto$(a,b)$ distribution has density $p(w \mid a,b) \propto w^{-b-1},\; w \in[a,\infty)$, where $a>0, b>0$. It is zero near the origin as the support is $[a,\infty)$.
Condition \eqref{hh} holds if $b > d/2$.

\noindent\textbf{Faster than polynomial near the origin}

A generalized inverse Gaussian distribution
$\mbox{GIG}(a,b,q)$ with density 
\[
p(w \mid a, b, q) \propto w^{q-1} \exp \left( -\frac{aw + b/w}{2} \right),
\]
where $a>0, b>0, q\in \mathbb{R}$, is  faster than polynomial near the origin.  
Condition \eqref{hh} holds for any GIG distribution.
When the mixing distribution is GIG, the distribution of the error is called Generalized Hyperbolic \citep{barndorff1982normal}.

The density of an inverse gamma distribution $\mbox{IG}(a,b)$ is $p(w \mid a,b) \propto w^{-a-1}\exp({-b/w})$, where  $a>0, b>0$.
$\mbox{IG}(a,b)$ is faster than polynomial near the origin.
Condition \eqref{hh} is satisfied if $a > d/2$.

For $\mu \in \mathbb{R}$ and $v > 0$, the $\mbox{Log-normal}(\mu, v)$ distribution has density
\[
p(w\mid \mu, v) \propto \frac{1}{w} \exp \left\{-\frac{(\ln w-\mu)^2}{2v^2} \right\}.
\]
This distribution is faster than polynomial near the origin and Condition \eqref{hh} holds.

A Fréchet distribution with the shape $\alpha > 0$ and scale $s > 0$   is given by 
\[
p(w\mid \alpha, s) \propto w^{-(1+\alpha)} \exp\left\{- \left({s}/{w} \right)^{\alpha} \right\}.
\]
It is faster than polynomial near the origin.
Moreover, Condition \eqref{hh} holds whenever $\alpha>d/2$.

\noindent\textbf{Non-zero near the origin and polynomial near the origin}

A $\mbox{Gamma}(a,b)$ distribution has density $p(w \mid a,b) \propto w^{a-1} \exp({-bw})$, where $a>0$ and $b>0$.
$\mbox{Gamma}(a,b)$ is polynomial near the origin with power $c = a-1$. 
The power
$c > c_1$ if
\[
a > \frac{n-p+m-\min\{d_1,\dots,d_n\}}{2} + 1,
\]
where $c_1$ is given in Theorem~\ref{thm:da}.
Condition \eqref{hh} always holds for Gamma distributions.
In particular, when $a=b=v/2$, the error has multivariate $t$ distribution with degrees of freedom $v$, and $c > c_1$ if
\[
v > n-p+m-\min\{d_1,\dots,d_n\} + 2.
\]


The $\mbox{Beta}(a,b)$ has density $p(u\mid a,b) \propto u^{a-1} (1-x)^{b-1},\ u\in(0,1)$, where $a>0$ and $b>0$. 
When $b=1$, the error is called multivariate slash distribution \citep{lange1993normal,1972}.
$\mbox{Beta}(a,b)$ is polynomial near the origin with power $c = a-1$. 
Condition \eqref{hh} always holds for Beta distributions.

The $\mbox{Weibull}(a,b)$ distribution has density  $p(u\mid a,b) \propto u^{a-1}\exp\{-(u/b)^a\}$, where $a>0$ and $b>0$. $\mbox{Weibull}(a,b)$ is polynomial near the origin with power $c = a-1$. 
Condition \eqref{hh} always holds for Weibull distributions.

The $F(a,b)$ distribution has density $p(w \mid a,b) \propto w^{a/2-1}(aw +b)^{-(a+b)/2}$, where $a>0$ and $b>0$.
$F(a,b)$ is polynomial near the origin with power $c = a/2-1$. The power
$c > c_1$ if
\[
a > n - p + m - \min \{d_1,\dots,d_n\} + 2.
\]
Condition \eqref{hh} is satisfied if $b > d$.

\subsection{Harris ergodicity of the DAI algorithm}\label{ssec:harris}




Now we give sufficient conditions for the DAI algorithm to be Harris ergodic.

\begin{theorem}\label{thm:dai}
Assume that Condition \eqref{hh} holds.
Let $\bm{k}$ be a missing structure.
Suppose that there is a missing structure $\bm{k}'$ and a realized value of $\bm{Y}_{(\bm{k}')}$ denoted by~$\bm{z}$ such that $\bm{k}' \prec \bm{k}$, and that $\pi_{\bm{k}'}(\cdot \mid \bm{z})$, the conditional density of $(\bm{B}, \bm{\Sigma})$ given $\bm{Y}_{(\bm{k}')} = \bm{z}$, is proper.
Then, for Lebesgue almost every~$\bm{z}'$ in the range of $\bm{Y}_{(\bm{k} - \bm{k}')}$, the posterior density $\pi_{\bm{k}}(\cdot \mid \bm{y}_{(\bm{k})})$, where~$\bm{y}$ satisfies $\bm{y}_{(\bm{k}')} = \bm{z}$ and $\bm{y}_{(\bm{k} - \bm{k}')} = \bm{z}'$, is proper, and any DAI chain targeting this posterior is Harris recurrent.
\end{theorem}

\begin{proof}
The posterior density $\pi_{\bm{k}}(\cdot \mid \bm{y}_{(\bm{k})})$ can be regarded as the posterior density of $(\bm{\beta}, \bm{\Sigma})$ given $\bm{Y}_{(\bm{k} - \bm{k}')} = \bm{z}'$ when the prior density is $\pi_{\bm{k}'}(\cdot \mid \bm{z})$.
Since $\pi_{\bm{k}'}(\cdot \mid \bm{z})$ is assumed to be proper, this posterior is proper for almost every possible value of $\bm{z}'$.

Under Condition \eqref{hh} and posterior propriety, a simple application of Theorem 6(v) of \cite{roberts2006harris} shows that a DAI chain targeting the posterior is Harris recurrent.
\end{proof}

By Theorem~\ref{thm:da}, a posterior density $\pi_{\bm{k}'}(\cdot \mid \bm{y}_{(\bm{k}')})$ is proper if each of the following conditions holds:
\begin{itemize}
\item $\bm{k}'$ is monotone, and Condition \eqref{da-condition} holds for $(\bm{k}', \bm{y}_{(\bm{k}')})$;
\item $\mix(\cdot)$ satisfies Condition \eqref{hh}; or
\item $\mix(\cdot)$ is either zero near the origin, or faster than polynomial near the origin, or polynomial near the origin with power $c > c_1$, where
\[
c_1=\frac{n-p+m-\min\{d'_1,\dots,d'_n\}}{2},
\] and for $i\in\{1,\dots,n\}$, $d'_i$ is the number of nonzero entries in the $i$th row of $\bm{k}'$.
\end{itemize}
By Theorem~\ref{thm:dai}, under Condition \eqref{hh}, to ensure Harris ergodicity of a DAI algorithm targeting $\pi_{\bm{k}}(\cdot \mid \bm{y}_{(\bm{k})})$ in an almost sure sense, one only needs to find a missing structure $\bm{k}' \prec \bm{k}$ that satisfies the conditions above.
In theory, when such a $\bm{k}'$ exists, it is still possible that the posterior $\pi_{\bm{k}}(\cdot \mid \bm{y}_{(\bm{k})})$ is improper -- even though it only happens on a zero measure set.
See \cite{fernandez1999multivariate} for a detailed discussion on this subtlety.


\subsection{Comparison between the DA and DAI algorithms}\label{ssec:compare}
Again, fix a missing structure $\bm{k}$ and a realized response $\bm{y}$, and let $d_i$ be the number of nonzero elements in the $i$th row of $\bm{k}$.
Assume that there is at least one missing entry, i.e., $d_i < d$ for some $i \in \{1,\dots,n\}$.
Assume also that $\pi_{\bm{k}}(\cdot \mid \bm{y}_{(\bm{k})})$ is proper.
In principle, one can either use the DA algorithm (with or without post hoc imputation) or the DAI algorithm associated with $(\bm{y}, \bm{k},\bm{k}')$ where $\bm{k} \prec \bm{k}'$ to sample from the posterior.
In this subsection, we compare the two algorithms in terms of their $L^2$ convergence rates.
This comparison is important when $\bm{k}$ is monotone, and Conditions \eqref{da-condition} and \eqref{hh} hold.
In this case, both algorithms can be efficiently implemented.

Let~$S$ and~$T$ be two random elements, defined on the measurable spaces $(\mathcal{X}_1, \mathcal{F}_1)$ and $(\mathcal{X}_2, \mathcal{F}_2)$ respectively.
Denote by~$\pi$ the marginal distribution of~$S$.
A generic data augmentation algorithm for sampling from~$\pi$ simulates a Markov chain $(S(t))_{t=0}^{\infty}$ that is reversible to~$\pi$.
Given the current state $S(t)$, the next state $S(t+1)$ is generated through the following procedure.
\begin{enumerate}
\item \textbf{I step.} Draw $T^*$ from the conditional distribution of~$T$ given $S = S(t)$.
Call the observed value $t^*$.

\item \textbf{P step.} Draw $S(t+1)$ from the conditional distribution of~$S$ given $T = t^*$.
\end{enumerate}
The DA and DAI algorithms for Bayesian robust linear regression are special cases of the above method.
Indeed, let~$S_1$,~$T_1$, and~$T_2$ be three random elements such that
the joint distribution of~$S_1$,~$T_1$, and~$T_2$ is the conditional joint distribution of $(\bm{B},\bm{\Sigma})$,~$\bm{W}$, and $\bm{Y}_{(\bm{k}'-\bm{k})}$ given $\bm{Y}_{(\bm{k})} = \bm{y}_{(\bm{k})}$.
Then, taking $S = S_1$ and $T=T_1$ in the generic algorithm yields the DA algorithm; taking $S=S_1$ and $T = (T_1,T_2)$ yields the DAI algorithm associated with $(\bm{y}, \bm{k},\bm{k}')$.

The $L^2$ convergence rate of the generic data augmentation chain is precisely the squared maximal correlation between~$S$ and~$T$ \citep{liu1994covariance}.
The maximal correlation between~$S$ and~$T$ is
\[
\gamma(S,T) := \sup \mbox{corr} [ f(S), g(T) ],
\]
where $\mbox{corr}$ means linear correlation, and the supremum is taken over real functions~$f$ and~$g$ such that the variances of $f(S)$ and $g(T)$ are finite.
Evidently,
\[
\gamma(S_1, T_1) \leq \gamma(S_1, (T_1,T_2)).
\]
We then have the following result.

\begin{theorem} \label{thm:speed}
Suppose that $\pi_{\bm{k}}(\cdot \mid \bm{y}_{(\bm{k})})$ is proper, and that the conditional distributions in the DA and DAI algorithms are well-defined.
Then, the $L^2$ convergence rate of the DA chain targeting  $\pi_{\bm{k}}(\cdot \mid \bm{y}_{(\bm{k})})$ is at least as small as that of any DAI chain targeting $\pi_{\bm{k}}(\cdot \mid \bm{y}_{(\bm{k})})$.
\end{theorem}

Recall that a smaller convergence rate means faster convergence.
Thus, when computation time is not considered and the observed missing structure is monotone, the DA algorithm is faster than the DAI algorithm. 
In this case, imputation of missing data, if needed, should be performed in a post hoc rather than intermediate manner.
In Section~\ref{sec:num} we use numerical experiments to show that this appears to be the case even after computation cost is taken into account.

\section{Numerical Experiment}\label{sec:num}

We compare the performance of the DA and DAI algorithms using simulated data.
All simulations are implemented through the \texttt{Bayesianrobust} R package.

Suppose that we have $n = 50$ observations in a study.
Assume that for each observation, the response has $d = 2$ components, while the predictor has the form $\bm{x}_i = (1,x_i)^\top$ where $x_i \in \mathbb{R}$.
We generate the $x_i$s using independent normal distributions.
The response matrix~$\bm{y}$ is generated according to the robust linear regression model, with the mixing distribution being $\mbox{Gamma}(1,1)$. 
The simulated data set is fixed throughout the section.

On the modeling side, we consider an independence Jeffreys prior with $m=d =2$ and $\bm{a}=\bm{0}$ (see~Equation \eqref{prior}) and three mixing distributions:

\begin{itemize}
\item G: The mixing distribution is Gamma$(2,2)$. The error is $t$ distribution with degrees of freedom 4.  
By Equation \eqref{eq:wconditional}, in the I step of the DA algorithm, one draws~$n$ independent Gamma random variables.
\item GIG: The mixing distribution is GIG$(1,1,-0.5)$. The error is generalised hyperbolic. 
By Equation  \eqref{eq:wconditional}, in the I step of the DA algorithm, one draws~$n$ independent generalized inverse Gaussian random variables.
\item  P: The mixing distribution is the point mass at~1.
The error is multivariate normal distribution. 
In this case, the posterior can be exactly sampled by the DA algorithm. 
\end{itemize}

We study three realized missing structures. 
Under these missing structures, the response matrix $\bm{y}$ has, respectively, 45, 40, and 35 rows fully observed.
The other rows all have only the second entry observed.
It is clear that all three missing structures are monotone.

In total, we consider nine combinations of mixing distributions and missing structures.
We apply both the DA and DAI algorithms in each scenario.
In the I2 step of the DAI algorithm (see Section~\ref{ssec:dai}), $\bm{k}'$ is taken to be the $n \times d$ matrix whose entries are all~1, i.e., $\bm{k}'$ corresponds to the case that the response matrix is fully observed.
At the end of each iteration of the DA algorithm, a post hoc imputation step is performed, so both algorithms impute all missing response entries.

Consider estimating the regression coefficients $\bm{B}$ and scatter matrix $\bm{\Sigma}$ using the posterior mean computed via the DA or DAI algorithm.
We compare the efficiency of the two algorithms based on the effective sample size (ESS) of each component separately and all components jointly.
At a given simulation length~$N$, the ESS of an MCMC estimator is defined to be~$N$ times the posterior variance divided by the asymptotic variance of the estimator \citep{vats2019multivariate}.
To account for computation cost, we also consider the ESS per minute, $\mbox{ESSpm} = \mbox{ESS}/t_N$, where $t_N$ is the number of minutes needed to run~$N$ iterations of the algorithm.
The simulation length is set to be $N = 30,000$ without burn-in, and the initial values are the ordinary least squares estimates using the observations that belong to pattern $1$, i.e., the observations without missing elements. 

\begin{figure}[!htb]\centering
\includegraphics[width=1\linewidth]{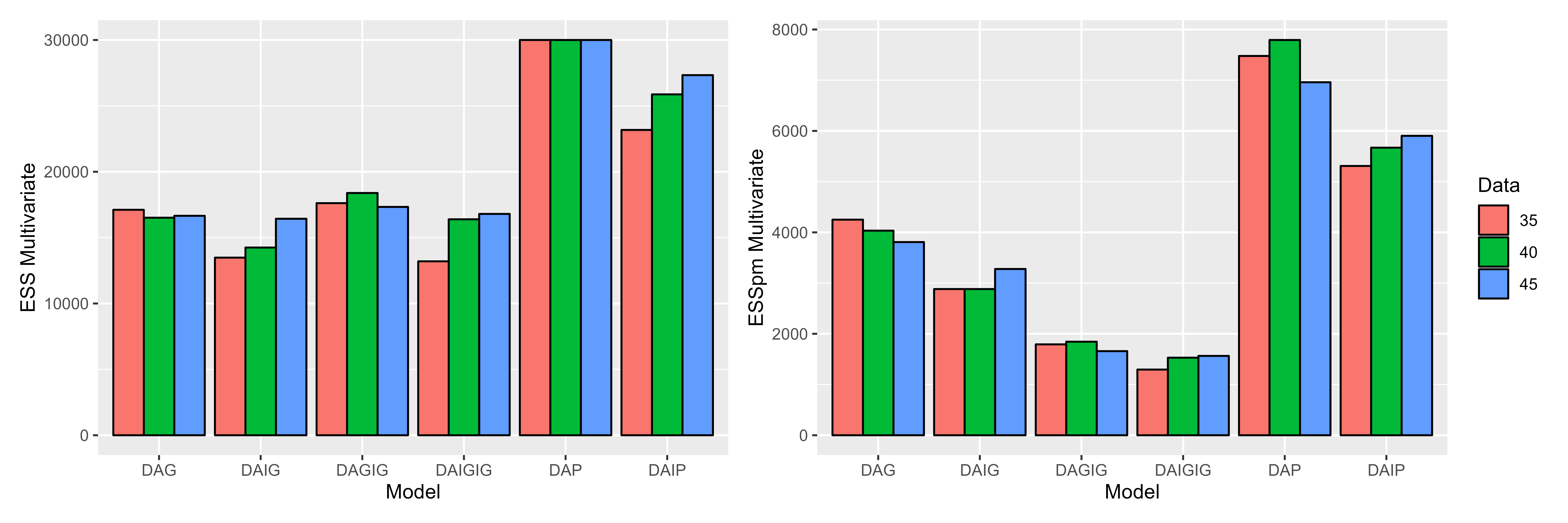}
\caption{The ESS and ESSpm of all components jointly in different scenarios. The $x$ axis lists the mixing distributions and samplers used, and the legend labels the missing structures by the number of fully observed rows.}\label{fig:exp2}
\end{figure}


Figure~\ref{fig:exp2} gives the ESS and ESSpm of all components jointly for each of the $3 \times 3 \times 2$ combinations of mixing distributions, missing structures, and MCMC algorithms. The ESS and ESSpm of each component separately are included in Table  \ref{app:exp2ess} and \ref{app:exp2esspm} in Appendix \ref{app:ess}.
We see that the DA algorithm gives a larger ESS compared to the DAI algorithm.
For the DAI algorithm, the ESS is lower when there are more missing data.
Similar trends appear when we consider the ESSpm.
In short, imputation of  missing responses values slows an algorithm down.
This is consistent with our theoretical results.
(Strictly speaking, our theoretical results concern convergence rate, not effective sample size, but for data augmentation chains which are reversible and positive, these two concepts are closely related.
See, e.g., \cite{rosenthal2003asymptotic}, Section 3.)

In addition to the experiment above, we consider another series of mixing distributions, namely $\mbox{Gamma}(v,v)$ with $v = 2,8,25$.
The resultant error distributions are multivariate $t$ distributions with $2v$ degrees of freedom.
Applying the DA and DAI algorithms to these models, we obtain Figure~\ref{fig:exp3}, the ESS and ESSpm of all components jointly. The ESS and ESSpm of each component separately are in Table \ref{app:exp3ess} and \ref{app:exp3esspm} in Appendix \ref{app:ess}.
Our simulation shows that when the model assumes that the error distribution has a lighter tail, the MCMC algorithms tend to be more efficient.

\begin{figure}[!htb]\centering
\includegraphics[width=1\linewidth]{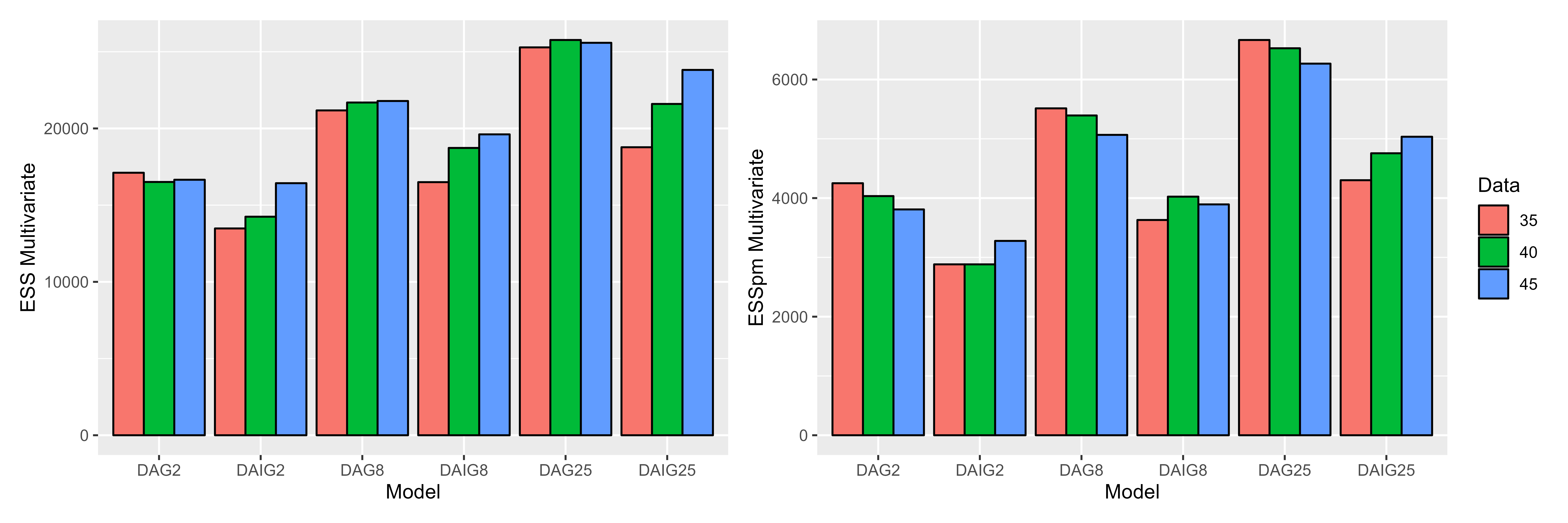}
\caption{The ESS and ESSpm of all components jointly in different scenarios. The $x$ axis lists the Gamma$(v,v)$ mixing distributions and samplers used, and the legend labels the missing structures by the number of fully observed rows.}\label{fig:exp3}
\end{figure}
%
%
%

\section{Conclusion}\label{sec:con}

We conducted convergence analysis for a data augmentation algorithms used in Bayesian robust multivariate linear regression with incomplete data. 
The algorithm was first proposed by \cite{liu1996bayesian}.
But, previously, little was known about its theoretical properties when the response matrix contains missing values.
We consider two versions of the algorithm, DA and DAI.
The DA algorithm can only be applied when the missing structure is monotone, whereas the DAI algorithm can be implemented for an arbitrary missing structure.
We establish geometric ergodicity of the DA algorithm under simple conditions.
For the DAI algorithm, we give conditions for Harris ergodicity. 
We compare the $L^2$ convergence rates of the DA and DAI algorithms.  
The $L^2$ convergence rate of the DA algorithm is at least as small as a corresponding DAI algorithm. 
A numerical study is provided. 
Under monotone missing structures, the DA algorithm outperforms the DAI computationally and theoretically.

\section*{Acknowledgments}

The first and second authors were supported by the National Science Foundation [grant number DMS-2112887], and the third author was supported by the National Science Foundation [grant number DMS-2152746].

\vspace{1cm}
\noindent{\bf\LARGE Appendix}

\appendix

\section{Details of the DA algorithm} \label{app:dadetails}

The DA algorithm consists of two steps, the I step and the P step. 
The I step is shown in Section~\ref{ssec:da}. 
In this section, we describe the P step when the missing structure $\bm{k}$ is monotone and Condition \eqref{da-condition} holds.

Recall from Section \ref{ssec:missing} that the monotone missing structure $\bm{k}$ has $d$ possible patterns. For $\ell \in \{1,\dots,d\}$, $n_{\ell(\bm{k})}$ is the number of observations that belong to pattern $\ell$. $\bm{x}_{(\bm{k},\ell)}$ and $\bm{y}_{(\bm{k},\ell)}$ are submatrices of $\bm{x}$ and $\bm{y}$ respectively defined in Section \ref{ssec:missing}. 
Define an $n \times n$ diagonal matrix $\bm{\lambda}_{\ell} = \mbox{diag}(w_1,\dots, w_{\sum_{j=1}^\ell n_{j(\bm{k})}}, 0, \dots, 0)$.
That is, the first $\sum_{j=1}^\ell n_{j(\bm{k})}$ diagonal entries of $\bm{\lambda}_{\ell}$ are  $w_1, \dots, w_{\sum_{j=1}^\ell n_{j(\bm{k})}}$, listed in order, and all the other entries are 0. 
Similarly, let  $\bm{\lambda}'_{\ell} = \mbox{diag}(w_1,\dots, w_{\sum_{j=1}^\ell n_j(\bm{k})})$. Denote $\hat{\betareal}_\ell$ as the weighted least squares estimates of the regression coefficient of $\bm{y}_{(\bm{k},\ell)}$ on  $\bm{x}_{(\bm{k},\ell)}$ with weight $(w_1,\dots, w_{\sum_{j=1}^\ell n_{j(\bm{k})}})$, and $\bm{s}_\ell$ as the corresponding weighted total sum of residual squares and cross-products matrix, i.e., 
\[
\hat{\betareal}_\ell = \left(\bm{x}^\top \bm{\lambda}_\ell \bm{x} \right)^{-1} \bm{x}_{(\bm{k},\ell)}^\top \bm{\lambda}'_\ell \bm{y}_{(\bm{k}, \ell)}, \quad \bm{s}_\ell = \left( \bm{y}_{(\bm{k},\ell)} - \bm{x}_{(\bm{k},\ell)} \hat{\betareal}_\ell \right)^\top \bm{\lambda}'_\ell \left( \bm{y}_{(\bm{k},\ell)} - \bm{x}_{(\bm{k},\ell)} \hat{\betareal}_\ell \right).
\] 
Let $\bm{a}_\ell$ be the lower right $(d-\ell+1)\times(d-\ell+1)$ submatrix of the positive semi-definite matrix $\bm{a}$ given in the prior (Equation \eqref{prior}). 

\textbf{P step.} Draw $(\bm{B}(t+1), \bm{\Sigma}(t+1))$ from the conditional distribution of $(\bm{B}, \bm{\Sigma})$ given $(\bm{W}, \bm{Y}_{(\bm{k})}) = (\bm{w}, \bm{y}_{(\bm{k})})$ using the following procedure.

\begin{enumerate}
	\item Draw $\bm{\Sigma}(t+1)$ given  $(\bm{W}, \bm{Y}_{(\bm{k})}) = (\bm{w}, \bm{y}_{(\bm{k})})$:

	For $\ell \in \{1, \dots, d\}$, let $\bm{c}_\ell = \bm{a}_\ell+ \bm{s}_\ell$, and let $\bm{e}_\ell$ be the lower triangular Cholesky factor of $\bm{c}_\ell^{-1}$ (so that $\bm{c}_\ell^{-1} = \bm{e}_\ell \bm{e}_\ell^\top$). 
	Draw a sequence of random vectors $\bm{F}_1,\dots, \bm{F}_d$ such that $\bm{F}_\ell = (F_{\ell,\ell},\dots,F_{d,\ell})^\top$ is $(d-\ell+1)$ dimensional,
	and that		
	\begin{enumerate}			
		\item $F_{i j} \sim N(0,1)$ for $1 \leq j < i\leq d$,		
		\item $F_{i i}^2 \sim \chi^2 (df_i)$ for $ i \in \{1,\dots, d\}$, where $df_i = (\sum_{j=1}^i n_j) -i+m-p-d+1$,		
		\item $F_{i j}$ are independent $\text{for all } 1 \leq j \leq i \leq d$.
	\end{enumerate}
	Denote the sampled values by $\bm{f}_1, \dots, \bm{f}_d$.
	
	
	Let $\bm{h}_\ell= \bm{e}_\ell\bm{f}_\ell$ for $\ell \in \{1, \dots, d\}$,   and let $\bm{h}$
	be a $d\times d$ lower triangular matrix with its lower triangular non-zero part formed by columns $\bm{h}_1,\dots,\bm{h}_d$. 
	Then $\sigmareal = (\bm{hh}^\top)^{-1}$ serves as a sampled value of $\bm{\Sigma}(t+1)$.	
	

	\item 
	
	Draw $\bm{B}(t+1)$ given $(\bm{\Sigma}(t+1),\bm{W},\bm{Y}_{(\bm{k})}) = (\bm{\sigmareal},\bm{w},\bm{y}_{(\bm{k})})$:
	
	%
	%
	
	Independently, draw $p$-dimensional standard normal random vectors $\bm{Z}_1, \dots, \bm{Z}_d$, and call the sampled values $\bm{z}_1, \dots, \bm{z}_d$.
	For $\ell \in \{1,\dots,d\}$, let $\bm{u}_\ell$ be the lower triangular Cholesky factor of $(\bm{x}^\top\bm{\lambda}_\ell \bm{x})^{-1}$.
	Then 
	\[
	(\bm{u}_1 \bm{z}_1,\dots, \bm{u}_d \bm{z}_d)\bm{h}^{-1} + (\hat{\betareal}_1 \bm{h}_{1}, \dots, \hat{\betareal}_d \bm{h}_{d}) \bm{h}^{-1}
	\]
	serves as a sampled value of $\bm{B}(t+1)$.

\end{enumerate}

Let us quickly consider a special case.
Recall that $\bm{k}_0 \in \{0,1\}^{n \times d}$ is the missing structure that corresponds to a completely observable response, i.e., all elements of~$\bm{k}_0$ are~1.
Then $\bm{Y}_{(\bm{k}_0)} = \bm{Y}$, $\bm{y}_{(\bm{k}_0)} = \bm{y}$, and~$\bm{k}_0$ is monotone.
Let $\bm{y}:\bm{x}$ be the matrix obtained by attaching~$\bm{x}$ to the right of~$\bm{y}$.
Then Condition \eqref{da-condition} for $(\bm{k}_0,\bm{y}_{(\bm{k}_0)})$ is equivalent to
\begin{equation} \label{da-condition-0} \tag{H1.0}
	r(\bm{y}: \bm{x}) = p + d, \quad n > p + 2d - m - 1. 
\end{equation}
Note that if there is some monotone~$\bm{k}$ such that Condition \eqref{da-condition} holds for $(\bm{k},\bm{y}_{(\bm{k})})$, then Condition \eqref{da-condition-0} necessarily holds.
Under~Condition \eqref{da-condition-0}, the conditional distribution of $(\bm{B},\bm{\Sigma})$ given $(\bm{W}, \bm{Y}) = (\bm{w}, \bm{y})$ is proper and rather simple.
We say $\bm{Z} \sim \mbox{N}_{p,d}(\bm{\mu},\bm{u},\bm{v})$ for $\bm{\mu} \in \mathbb{R}^{p \times d}$, $\bm{u} \in S_+^{p \times p}$, and $\bm{v} \in S_+^{d \times d}$ if $\bm{Z}$ is a $p \times d$ random matrix associated with the probability density function
\[
\frac {\exp \left[-{\frac {1}{2}}\,\mathrm {tr} \left\{\bm{v} ^{-1}(\bm{z} -\bm{\mu} )^{T} \bm{u} ^{-1}(\bm{z} -\bm{\mu} )\right\}\right]}{(2\pi )^{pd/2}|\bm{v} |^{p/2}|\bm{u} |^{d/2}}, \quad \bm{z} \in \mathbb{R}^{p \times d}.
\]
$\mbox{N}_{p,d}$ is called a matrix normal distribution.
We say $\bm{Z} \sim \mbox{IW}_d(\nu, \bm{\psi})$ for $\nu > d-1$ and $\bm{\psi} \in S_+^{d \times d}$ if $\bm{Z}$ is a $d \times d$ random matrix associated with the probability density function
\[
\frac{|\bm{\psi}|^{\nu/2}}{2^{\nu d/2} \Gamma_d(\nu/2)} |\bm{z}|^{-(\nu + d + 1)/2} \exp \left\{ - \frac{1}{2} \mbox{tr}\left( \bm{\psi} \bm{z}^{-1} \right) \right\}, \quad \bm{z} \in S_+^{d \times d},
\]
where $\mbox{tr}(\cdot)$ returns the trace of a matrix, and $\Gamma_d(\cdot)$ is a multivariate gamma function.
$\mbox{IW}_d$ is called an inverse Wishart distribution, and it is well-known that a random matrix follows the $\mbox{IW}_d(\nu, \bm{\psi})$ distribution if and only if its inverse follows the Wishart distribution $\mbox{W}_d(\nu, \bm{\psi}^{-1})$, which has density
\[
\frac{|\bm{\psi}|^{\nu/2}}{2^{\nu d/2} \Gamma_d(\nu/2) } |\bm{z}|^{(\nu-d-1)/2} \exp \left\{ -\frac{1}{2} \mbox{tr} \left( \bm{\psi} \bm{z} \right) \right\}, \quad \bm{z} \in S_+^{d \times d}.
\]
For $\bm{w} = (w_1,\dots,w_n) \in (0,\infty)^n$, let $\bm{\lambda} = \mbox{diag}(w_1,\dots,w_n)$, i.e., the diagonal matrix whose diagonal elements are $w_1,\dots,w_n$, listed in order.
Let
\[
\hat{\betareal} = \left(\bm{x}^\top \bm{\lambda} \bm{x} \right)^{-1} \bm{x}^\top \bm{\lambda} \bm{y}, \quad \bm{s} = \left( \bm{y} - \bm{x} \hat{\betareal} \right)^\top \bm{\lambda} \left( \bm{y} - \bm{x} \hat{\betareal} \right),
\] 
and $\bm{\xi} =  (\bm{x}^\top \bm{\lambda} \bm{x})^{-1}$.
Then it is easy to see from~Equation \eqref{eq:fullmodel} that
\[
\begin{aligned}
	\bm{\Sigma} \mid (\bm{W}, \bm{Y}) = (\bm{w}, \bm{y}) &\sim \mbox{IW}_d \left( n - p + m - d, \bm{s} + \bm{a} \right), \\
	\bm{B} \mid (\bm{\Sigma}, \bm{W}, \bm{Y}) = (\sigmareal, \bm{w}, \bm{y})  &\sim \mbox{N}_{p,d} \left( \hat{\betareal}, \bm{\xi}, \sigmareal \right).
\end{aligned}
\]

\section{Proof of Theorem~\ref{thm:da}}\label{app:dageo}

We prove posterior propriety and geometric ergodicity by establishing drift and minorization (d\&m) conditions, which we now describe. 
Let $(\mathcal{X}, \mathcal{F})$ be a measurable space.
Consider a Markov chain $(X(t))_{t=0}^{\infty}$ whose state space is $(\mathcal{X}, \mathcal{F})$, and let $K: \mathcal{X} \times \mathcal{F} \to [0,1]$ be its transition kernel.
We say that $(X(t))$ satisfies a d\&m condition with drift function~$V: \mathcal{X} \to [0,\infty)$, minorization measure $\nu: \mathcal{F} \to [0,1]$, and parameters $(\eta,\varrho,\delta,\epsilon) \in \mathbb{R}^4$ if each of the following two conditions holds:
\begin{description}
	\item[Drift condition:] 
	For $x \in \mathcal{X}$,
	\[
	KV(x) := \int_{\mathcal{X}} V(x') K(x,\df x') \leq \eta V(x) + \varrho,
	\]
	where $\eta < 1$ and $\varrho<\infty$. Note that $KV(x)$ can be interpreted as the conditional expectation of $V(X(t+1))$ given $X(t) = x$, where~$t$ can be any non-negative integer.
	
	\item[Minorization condition:]
	Let $\nu$ be a probability measure, $\epsilon > 0$, and $\delta > 2\varrho/(1-\eta)$.
	Moreover, whenever $V(x) < \delta$, 
	\[
	K(x,A) \geq \epsilon \nu(A)\; \text{ for each }A \in \mathcal{F}.
	\]
\end{description}
It is well-known that if a d\&m condition holds, then the Markov chain has a proper stationary distribution, and it is geometrically ergodic \citep{rosenthal1995minorization}.
See also \cite{jones2001honest},  \cite{roberts2004general}, and \cite{meyn2005markov}.


We begin by establishing a minorization condition for the DA algorithm associated with a realized response $\bm{y}$ and missing structure~$\bm{k}$.
As before, $d_i$ will be used to denote the number of nonzero elements in the $i$th row of~$\bm{k}$.
Let $( \bm{B}(t), \bm{\Sigma}(t))_{t=0}^{\infty}$ by the underlying Markov chain, and denote its Markov transition kernel by~$K$.
Set $\mathcal{X} = \mathbb{R}^{p \times d} \times S_+^{d \times d}$, and let~$\mathcal{F}$ be the usual Borel algebra associated with~$\mathcal{X}$.
Define a drift function
\[
V(\betareal,\sigmareal) = \sum_{ i = 1 }^n r_{i,(\bm{k})},
\]
where, for $i \in \{1,\dots,n\}$,
\[
r_{i,(\bm{k})} = \left(\bm{y}_{i, (\bm{k})} - \bm{c}_{i,(\bm{k})} \betareal^\top \bm{x}_i \right)^\top \left( \bm{c}_{i,(\bm{k})} \sigmareal \bm{c}_{i,(\bm{k})}^\top \right)^{-1}  \left(\bm{y}_{i, (\bm{k})} - \bm{c}_{i,(\bm{k})} \betareal^\top \bm{x}_i \right).
\]
Then the following holds.
\begin{lemma} \label{lem:minor}
	For any $\delta > 0$, there exist a probability measure $\nu: \mathcal{F} \to [0,1]$ and $\epsilon > 0$ such that whenever $V(\betareal,\sigmareal) < \delta$,
	\[
	K((\betareal,\sigmareal),A) > \epsilon \nu(A)\; \text{ for each }A \in \mathcal{F}.
	\]
\end{lemma}
\begin{proof}
	One can write
	\[
	K((\betareal,\sigmareal),A) = \int_{(0,\infty)^n} Q_{{\bm{k}}}( A \mid \bm{w}, \bm{y}_{(\bm{k})} ) \, P_{{\bm{k}}}(\df \bm{w} \mid \betareal, \sigmareal, \bm{y}_{(\bm{k})}) ,
	\]
	where $P_{{\bm{k}}}(\cdot \mid \betareal, \sigmareal, \bm{y}_{(\bm{k})})$ is the conditional distribution of $\bm{W}$ given $(\bm{B}, \bm{\Sigma}, \bm{Y}_{(\bm{k})}) = (\betareal, \sigmareal, \bm{y}_{(\bm{k})})$, and $Q_{{\bm{k}}}(\cdot \mid \bm{w}, \bm{y}_{(\bm{k})} )$ is the conditional distribution of $(\bm{\beta},\bm{\Sigma})$ given $(\bm{W}, \bm{Y}_{(\bm{k})}) = (\bm{w}, \bm{y}_{(\bm{k})})$, both derived from~Equation \eqref{eq:fullmodel}.
	As stated in Section~\ref{ssec:da},
	\[
	\begin{aligned}
		P_{{\bm{k}}}(\df \bm{w} \mid \betareal,\sigmareal, \bm{y}_{(\bm{k})}) = \prod_{i=1}^n  \frac{ w_i^{d_i/2} \exp \left( - r_{i,(\bm{k})} w_i/2 \right) \, \mix( \df w_i)}{\int_0^{\infty} w^{d_i/2} \exp \left( - r_{i,(\bm{k})} w / 2 \right) \, \mix( \df w)  }, &\\
		\quad \bm{w} = (w_1,\dots,w_n)^\top \in (0,\infty)^n.&
	\end{aligned}
	\]
	Assume that $V(\betareal,\sigmareal) < \delta$ for some $\delta > 0$.
	Then $r_{i,(\bm{k})} < \delta$ for each~$i$.
	It follows that, for any measurable $A' \in (0,\infty)^n$,
	\[
	P_{{\bm{k}}}(A' \mid \betareal,\sigmareal, \bm{y}_{(\bm{k})}) \geq \epsilon \int_{A'} \prod_{i=1}^n \frac{ w_i^{d_i/2} \exp \left( - \delta w_i/2 \right) \, \mix( \df w_i)}{\int_0^{\infty} w^{d_i/2} \exp \left( - \delta w / 2 \right) \, \mix( \df w)  },
	\]
	where
	\[
	\epsilon = \left\{ \frac{\int_0^{\infty} w^{d_i/2} \exp \left( - \delta w / 2 \right) \, \mix( \df w)}{\int_0^{\infty} w^{d_i/2} \, \mix( \df w)} \right\}^n.
	\]
	Thus, for each $A \in \mathcal{F}$,
	\[
	K((\betareal,\sigmareal),A) \geq \epsilon \nu(A),
	\]
	where $\nu(\cdot)$ is a probability measure given by
	\[
	\nu(A) = \int_{(0,\infty)^n} Q_{{\bm{k}}}(A \mid \bm{w}, \bm{y}_{(\bm{k})}) \, \prod_{i=1}^n \frac{ w_i^{d_i/2} \exp \left( - \delta w_i/2 \right) \, \mix( \df w_i)}{\int_0^{\infty} w^{d_i/2} \exp \left( - \delta w / 2 \right) \, \mix( \df w)  }.
	\]
\end{proof}

To establish d\&m, it remains to show the following.
\begin{lemma}
	Suppose that Condition \eqref{hh} holds, and that the conditional distribution of $(\bm{B},\bm{\Sigma})$ given $(\bm{W}, \bm{Y}_{(\bm{k})}) = (\bm{w}, \bm{y}_{(\bm{k})})$ is proper for every $\bm{w} = (w_1,\dots,w_n)^\top \in (0,\infty)^n$.
	If any one of the following conditions holds, then there exist $\eta < 1$ and $\varrho < \infty$ such that, for $(\betareal,\sigmareal) \in \mathcal{X}$,
	\[
	KV(\betareal,\sigmareal) < \eta V(\betareal,\sigmareal) + \varrho.
	\]
	\begin{enumerate}
		\item $\mix(\cdot)$ is zero near the origin;
		\item $\mix(\cdot)$ is faster than polynomial near the origin; or
		\item $\mix(\cdot)$ is polynomial near the origin with power $c> c_1$, where
		\[
		c_1=\frac{n-p+m-\min\{d_1,\dots,d_n\}}{2}.
		\]
	\end{enumerate}
\end{lemma}

\begin{proof}
	
	We will prove the result when $\bm{k}$ contains at least one vanishing element.
	When 
	there are no missing data under~$\bm{k}$, the result is proved by \cite{hobert2018convergence}.
	(To be absolutely precise, Hobert et al. assumed that $\mix(\cdot)$ is absolutely continuous with respect to the Lebesgue measure in the ``zero near the origin" case, but their proof can be easily extended to the case where absolute continuity is absent.)
	
	For $\betareal \in \mathbb{R}^{p \times d}$ and $\sigmareal \in S_+^{d \times d}$, let $P_{{\bm{k}}}(\cdot \mid \betareal, \sigmareal, \bm{y}_{(\bm{k})})$ be the conditional distribution of $\bm{W}$ given $(\bm{B}, \bm{\Sigma}, \bm{Y}_{(\bm{k})}) = (\betareal, \sigmareal, \bm{y}_{(\bm{k})})$.
	For $\bm{w} \in (0,\infty)^n$, let $Q_{{\bm{k}}}(\cdot \mid \bm{w}, \bm{y}_{(\bm{k})})$ be the conditional distribution of $(\bm{B},\bm{\Sigma})$ given $(\bm{W}, \bm{Y}_{(\bm{k})}) = (\bm{w}, \bm{y}_{(\bm{k})})$.
	Then
	\begin{equation} \label{eq:KV}
		KV(\betareal,\sigmareal) = \int_{(0,\infty)^n} \int_{\mathbb{R}^{p \times d} \times S_+^{d \times d}} V(\betareal^*,\sigmareal^*)  \, Q_{{\bm{k}}}(\df \betareal^*, \df \sigmareal^* \mid \bm{w}, \bm{y}_{(\bm{k})} ) \, P_{{\bm{k}}}(\df \bm{w} \mid \betareal, \sigmareal, \bm{y}_{(\bm{k})}).
	\end{equation}

	Let $\bm{k}_0 \in \{0,1\}^{n \times d}$ be a matrix such that all its entries are~1.
	In other words, $\bm{K} = \bm{k}_0$ if there are no data missing.
	$\bm{Y}_{(\bm{k}_0 - \bm{k})}$ denotes the unobservable entries of $\bm{Y}$ under the missing structure~$\bm{k}$.
	By our assumptions, given $(\bm{W}, \bm{Y}_{(\bm{k})}) = (\bm{w}, \bm{y}_{(\bm{k})})$, $(\bm{B}, \bm{\Sigma}, \bm{Y}_{(\bm{k}_0 - \bm{k})})$ has a proper conditional distribution.
	For $\bm{w} \in (0,\infty)^n$, denote by $Q_{1,{\bm{k}}}(\cdot \mid \bm{w}, \bm{y}_{(\bm{k})})$ the conditional distribution of $\bm{Y}_{(\bm{k}_0 - \bm{k})}$ given  $(\bm{W}, \bm{Y}_{(\bm{k})}) = (\bm{w}, \bm{y}_{(\bm{k})})$.
	For $\bm{w} \in (0,\infty)^n$ and a realized value of $\bm{Y}_{(\bm{k}_0 - \bm{k})}$, say, $\bm{z} \in \mathbb{R}^{\sum_{i=1}^n (d-d_i)}$, let $Q_{2,{\bm{k}}}(\cdot \mid \bm{w}, \bm{z} , \bm{y}_{(\bm{k})})$ be the conditional distribution of $(\bm{B}, \bm{\Sigma})$ given $(\bm{W}, \bm{Y}_{(\bm{k}_0 - \bm{k})}, \bm{Y}_{(\bm{k})}) = (\bm{w}, \bm{z}, \bm{y}_{(\bm{k})})$.
	We now describe this distribution (see Appendix \ref{app:dadetails}).
	Let $\bm{y}^*$ be a realized value of~$\bm{Y}$ such that $\bm{y}_{(\bm{k})}^* = \bm{y}_{(\bm{k})}$.
	Let $\bm{\lambda} = \mbox{diag}(w_1,\dots,w_n)$, where $w_1,\dots,w_n$ are the components of~$\bm{w}$.
	Let
	\[
	\hat{\betareal} = \left(\bm{x}^\top \bm{\lambda} \bm{x} \right)^{-1} \bm{x}^\top \bm{\lambda} \bm{y}^*, \quad \bm{s} = \left( \bm{y}^* - \bm{x} \hat{\betareal} \right)^\top \bm{\lambda} \left( \bm{y}^* - \bm{x} \hat{\betareal} \right),
	\] 
	and $\bm{\xi} =  (\bm{x}^\top \bm{\lambda} \bm{x})^{-1}$.
	Then $Q_{2,\bm{k}}(\cdot \mid \bm{\omega}, \bm{y}_{(\bm{k}_0 - \bm{k})}^*, \bm{y}_{\bm{(k)}})$ is given by
	\[
	\begin{aligned}
		\bm{\Sigma} \mid (\bm{W}, \bm{Y}) = (\bm{w}, \bm{y}^*) &\sim \mbox{IW}_d \left( n - p + m - d, \bm{s} + \bm{a} \right), \\
		\bm{B} \mid (\bm{\Sigma}, \bm{W}, \bm{Y}) = (\sigmareal, \bm{w}, \bm{y}^*)  &\sim \mbox{N}_{p,d} \left( \hat{\betareal}, \bm{\xi}, \sigmareal \right).
	\end{aligned}
	\]
	It is easy to verify that $Q_{{\bm{k}}}$, $Q_{1,{\bm{k}}}$, and $Q_{2,{\bm{k}}}$ are connected through the following tower property:
	\begin{equation} \label{eq:tower}
		Q_{{\bm{k}}} \left(\cdot \mid \bm{w}, \bm{y}_{(\bm{k})} \right) = \int_{\mathbb{R}^{\sum_{i=1}^n (d-d_i)}} Q_{2,{\bm{k}}} \left(\cdot \mid \bm{w}, \bm{y}_{(\bm{k}_0 - \bm{k})}^*, \bm{y}_{(\bm{k})} \right) \, Q_{1,{\bm{k}}} \left(\df \bm{y}_{(\bm{k}_0 - \bm{k})}^* \mid \bm{w}, \bm{y}_{(\bm{k})} \right).
	\end{equation}
	
	In light of~Equations \eqref{eq:KV} and~\eqref{eq:tower}, let us first investigate the integral
	\begin{equation} \label{eq:P4V}
		\begin{aligned}
			&\int_{\mathbb{R}^{p \times d} \times S_+^{d \times d}} V(\betareal^*,\sigmareal^*) \, Q_{2,{\bm{k}}} \left(\df \betareal^*, \df \sigmareal^* \mid \bm{w}, \bm{y}_{(\bm{k}_0 - \bm{k})}^*, \bm{y}_{(\bm{k})} \right) \\
			=& \int_{S_+^{d \times d}} \int_{\mathbb{R}^{p \times d}} V(\betareal^*,\sigmareal^*) \, Q_{4,{\bm{k}}} \left(\df \betareal^* \mid \sigmareal^*, \bm{w}, \bm{y}_{(\bm{k}_0 - \bm{k})}^*, \bm{y}_{(\bm{k})} \right) \, Q_{3,{\bm{k}}} \left( \df \sigmareal^* \mid \bm{w},  \bm{y}_{(\bm{k}_0 - \bm{k})}^*, \bm{y}_{(\bm{k})}  \right),
		\end{aligned}
	\end{equation}
	where $Q_{3,{\bm{k}}}(\cdot \mid \bm{w}, , \bm{y}_{(\bm{k}_0 - \bm{k})}^*, \bm{y}_{(\bm{k})})$ is the $\mbox{IW}_d(n-p+m-d, \bm{s} + \bm{a})$ distribution, and $Q_{4,{\bm{k}}}(\cdot \mid \sigmareal^*, \bm{w}, \bm{y}_{(\bm{k}_0 - \bm{k})}^*, \bm{y}_{(\bm{k})})$ is the $\mbox{N}_{p,d}(\hat{\bm{\beta}}, \bm{\xi}, \sigmareal^*)$ distribution.
	It follows from basic properties of matrix normal distributions that
	\[
	\begin{aligned}
		&\int_{\mathbb{R}^{p \times d}} V(\betareal^*,\sigmareal^*) \, Q_{4,{\bm{k}}} \left(\df \betareal^* \mid \sigmareal^*, \bm{w}, \bm{y}_{(\bm{k}_0 - \bm{k})}^*, \bm{y}_{(\bm{k})} \right) \\
		=& \sum_{i=1}^n \left\{ \left( \bm{y}_{i,(\bm{k})} - \bm{c}_{i,(\bm{k})} \hat{\bm{\beta}}^\top \bm{x}_i \right)^\top \left( \bm{c}_{i,(\bm{k})} \sigmareal^* \bm{c}_{i,(\bm{k})}^\top \right)^{-1} \left( \bm{y}_{i,(\bm{k})} - \bm{c}_{i,(\bm{k})} \hat{\bm{\beta}}^\top \bm{x}_i \right) + d_i \bm{x}_i^\top \bm{\xi} \bm{x}_i\right\}.
	\end{aligned}
	\]
	By Theorem 3.2.11 of \cite{muirhead2009aspects}, if $\bm{\Sigma}^*$ is a random matrix that follows the $\mbox{IW}_d(n-p+m-d, \bm{s} + \bm{a})$ distribution, then $(\bm{c}_{i,(\bm{k})} \bm{\Sigma}^* \bm{c}_{i,(\bm{k})}^\top)^{-1}$ follows the $\mbox{W}_{d_i}(n-p+m-2d+d_i, [\bm{c}_{i,(\bm{k})} (\bm{s} + \bm{a}) \bm{c}_{i,(\bm{k})}^\top ]^{-1})$ distribution.
	Then by basic properties of Wishart distributions,
	\begin{equation} \label{eq:P5P6V}
		\begin{aligned}
			& \int_{S_+^{d \times d}} \int_{\mathbb{R}^{p \times d}} V(\betareal^*,\sigmareal^*) \, Q_{4,{\bm{k}}} \left(\df \betareal^* \mid \sigmareal^*, \bm{w}, \bm{y}_{(\bm{k}_0 - \bm{k})}^*, \bm{y}_{(\bm{k})} \right) \, Q_{3,{\bm{k}}} \left( \df \sigmareal^* \mid \bm{w}, , \bm{y}_{(\bm{k}_0 - \bm{k})}^*, \bm{y}_{(\bm{k})}  \right) \\
			=& \sum_{i=1}^n (n-p+m-2d+d_i)  \left[ \left( \bm{y}_{i,(\bm{k})} - \bm{c}_{i,(\bm{k})} \hat{\bm{\beta}}^\top \bm{x}_i \right)^\top \left\{ \bm{c}_{i,(\bm{k})} (\bm{s} + \bm{a}) \bm{c}_{i,(\bm{k})}^\top \right\}^{-1} \left( \bm{y}_{i,(\bm{k})} - \bm{c}_{i,(\bm{k})} \hat{\bm{\beta}}^\top \bm{x}_i \right) \right] + \\
			&\sum_{ i = 1 }^n d_i \bm{x}_i^\top \bm{\xi} \bm{x}_i .
		\end{aligned}
	\end{equation}
	For $i=1,\dots,n$, denote by $\bm{y}_i^*$ the $i$th row of~$\bm{y}^*$, and note that $ \bm{c}_{i, (\bm{k})} \bm{y}_i^* = \bm{y}_{i,(\bm{k})}$.
	It follows from Lemma \ref{lem:positive}, which is stated right after this proof, that, for $i = 1,\dots,n$,
	\[
	\bm{x}_i^\top \bm{\xi} \bm{x}_i =  \bm{x}_i^\top \left( \sum_{j=1}^n w_j \bm{x}_j \bm{x}_j^\top \right)^{-1} \bm{x}_i \leq \frac{1}{w_i},
	\]
	and
	\[
	\begin{aligned}
		&\left( \bm{y}_{i,(\bm{k})} - \bm{c}_{i,(\bm{k})} \hat{\bm{\beta}}^\top \bm{x}_i \right)^\top \left\{ \bm{c}_{i,(\bm{k})} (\bm{s} + \bm{a}) \bm{c}_{i,(\bm{k})}^\top \right\}^{-1} \left( \bm{y}_{i,(\bm{k})} - \bm{c}_{i,(\bm{k})} \hat{\bm{\beta}}^\top \bm{x}_i \right) \\
		=& \left( \bm{y}_{i,(\bm{k})} - \bm{c}_{i,(\bm{k})} \hat{\bm{\beta}}^\top \bm{x}_i \right)^\top \left[ \bm{c}_{i,(\bm{k})} \left\{ \sum_{ j = 1 }^n w_j \left(\bm{y}_j^* - \hat{\bm{\beta}}^\top \bm{x}_j \right) \left(\bm{y}_j^* - \hat{\bm{\beta}}^\top \bm{x}_j \right)^\top \right\} \bm{c}_{i,(\bm{k})}^\top + \bm{c}_{i,(\bm{k})} \bm{a} \bm{c}_{i,(\bm{k})}^\top \right]^{-1} \\
		&\left( \bm{y}_{i,(\bm{k})} - \bm{c}_{i,(\bm{k})} \hat{\bm{\beta}}^\top \bm{x}_i \right) \\
		=& \left( \bm{y}_{i,(\bm{k})} - \bm{c}_{i,(\bm{k})} \hat{\bm{\beta}}^\top \bm{x}_i \right)^\top \left\{\sum_{j=1}^n w_j \left( \bm{y}_{j,(\bm{k})} - \bm{c}_{j,(\bm{k})} \hat{\bm{\beta}}^\top \bm{x}_j \right) \left( \bm{y}_{j,(\bm{k})} - \bm{c}_{j,(\bm{k})} \hat{\bm{\beta}}^\top \bm{x}_j \right)^\top + \bm{c}_{i,(\bm{k})} \bm{a} \bm{c}_{i,(\bm{k})}^\top \right\} ^{-1} \\
		&\left( \bm{y}_{i,(\bm{k})} - \bm{c}_{i,(\bm{k})} \hat{\bm{\beta}}^\top \bm{x}_i \right) \\
		\leq & \frac{1}{w_i}.
	\end{aligned}
	\]
	
	Notice that $1/\omega_i$ is a constant, and does not depend on $\bm{y}_{(\bm{k}_0 - \bm{k})}^*$. 
	It then follows from~Equations \eqref{eq:KV} to \eqref{eq:P5P6V} that
	\[
	\begin{aligned}
		KV(\betareal,\sigmareal) &\leq \int_{(0,\infty)^n} \left( \sum_{ i = 1 }^n \frac{n-p+m-2d+2d_i}{w_i} \right) P_{{\bm{k}}}(\df \bm{w} \mid \betareal, \sigmareal, \bm{y}_{(\bm{k})})\\& \leq \int_{(0,\infty)^n} \left( \sum_{ i = 1 }^n \frac{n-p+m}{w_i} \right) P_{{\bm{k}}}(\df \bm{w} \mid \betareal, \sigmareal, \bm{y}_{(\bm{k})}),
	\end{aligned}
	\]
	where $P_{{\bm{k}}}(\df \bm{w} \mid \betareal, \sigmareal, \bm{y}_{(\bm{k})}$) is given by Equation \eqref{eq:wconditional}.
	By Lemma \ref{lem:weightintegral} below, there exist $\eta' < 1/(n-p+m)$ and $\varrho' < \infty$ such that, for $(\betareal, \sigmareal) \in \mathcal{X}$,
	$$ 
	\int_{(0,\infty)^n} \left( \sum_{ i = 1 }^n \frac{n-p+m}{w_i} \right) P_{{\bm{k}}}(\df \bm{w} \mid \betareal, \sigmareal, \bm{y}_{(\bm{k})})\\
	\leq \sum_{ i = 1 }^n (n-p+m)(\eta' {r_{i, (\bm{k})}} + \varrho') .
	$$
	Then desired result holds with $\eta = (n-p+m) \eta' < 1$ and $\varrho = n (n-p+m) \varrho' < \infty$.
\end{proof}

\begin{lemma}\label{lem:positive}
	Let $\bm{\varphi}$ be a positive definite matrix, and $\bm{\upsilon}$ be a vector, such that the matrix $\bm{\varphi} - \bm{\upsilon \upsilon}^{T}$ is positive semidefinite, then $\bm{\upsilon}^{T} \bm{\varphi}^{-1} \bm{\upsilon} \leq 1$.
\end{lemma}
For the proof of Lemma \ref{lem:positive}, see, e.g., \cite{roy2010monte}, Lemma 3.

\begin{lemma}\label{lem:weightintegral}
	Suppose that Condition \eqref{hh} holds and that $P_{\text{mix}}{(\cdot)}$ is either zero near the origin, or faster than polynomial near the origin, or polynomial near the origin with power $c> c_1$, where
	\[
	c_1=\frac{n-p+m-\min\{d_1,\dots,d_n\}}{2}.
	\]
	Then there exist $\eta  \in[0, 1/(n-p+m))$ and $\varrho  < \infty$, such that for all $\tilde{d} \in [\min\{d_1, \dots, d_n\}, d]$ and $\tilde{r}\in[0,\infty)$,
	\begin{equation}\nonumber
		\frac{\int_{0}^{\infty} w^{(\tilde{d}-2)/2} \exp({-\tilde{r}w/2}) \mix{(\df w)}}{\int_{0}^{\infty} w^{\tilde{d}/2} \exp({-\tilde{r}w/2}) \mix{(\df w)}} \leq \eta \tilde{r} + \varrho.  
	\end{equation}
\end{lemma}

\begin{proof}	
	We will make use of results in \cite{hobert2018convergence}.
	
	When $P_{\text{mix}}(\cdot)$ is zero near the origin, there exists $\theta>0$, such that $\int_{0}^{\theta}  P_{\text{mix}}(\df w)=0$. 
	Then, for all $\tilde{d} \in [\min\{d_1, \dots, d_n\}, d]$ and $\tilde{r}\in[0,\infty)$,
	\[
	\begin{aligned}
		\frac{\int_{0}^{\infty} w^{(\tilde{d}-2)/2} \exp({-\tilde{r}w/2}) P_{\text{mix}}{(\df w)}}{\int_{0}^{\infty} w^{\tilde{d}/2} \exp({-\tilde{r}w/2}) P_{\text{mix}}{(\df w)}} &= \frac{\int_{\theta}^{\infty} (1/w) w^{\tilde{d}/2} \exp({-\tilde{r}w/2}) P_{\text{mix}}{(\df w)}}{\int_{\theta}^{\infty} w^{\tilde{d}/2} \exp({-\tilde{r}w/2}) P_{\text{mix}}{(\df w)}} \\
		&\leq  \frac{(1/\theta)\int_{\theta}^{\infty} w^{\tilde{d}/2} \exp({-\tilde{r}w/2}) P_{\text{mix}}{(\df w)}}{\int_{\theta}^{\infty} w^{\tilde{d}/2} \exp({-\tilde{r}w/2}) P_{\text{mix}}{(\df w)}}\\
		&=1/{\theta}.
	\end{aligned}
	\]
	
	When $P_{\text{mix}}(\cdot)$ is polynomial near the origin or faster than polynomial near the origin, recall that $\mixdensity(\cdot)$ is the density function of $\mix(\cdot)$ with respect to the Lebesgue measure. 
	Let $S(\mixdensity)$ be the set of $\eta\in[0,\infty)$ such that there exists $\varrho_{\eta}\in\mathbb{R}$, satisfying 		
	$$\frac{\int_{0}^{\infty} w^{-1/2} \exp({-\tilde{r}w/2}) P_{\text{mix}}{(\df w)}}{\int_{0}^{\infty} w^{1/2} \exp({-\tilde{r}w/2}) P_{\text{mix}}{(\df w)}} \leq \eta \tilde{r} + \varrho_{\eta} $$
	for all $\tilde{r}\in [0,\infty)$. If $S(\mixdensity)$ is non-empty, let $\eta_{\mixdensity} = \inf S(\mixdensity) $; otherwise, set $\eta_{\mixdensity}=\infty$.

	Consider a Gamma $(\alpha, 1)$ mixing distribution with density 
	\[
	p_{G(\alpha)}(w) = \{\Gamma_1(\alpha)\}^{-1} w^{\alpha-1}\exp({-w}).
	\]
	It is easy to see that if $\alpha>1/2$, for all $\tilde{r}\in[0,\infty)$, 
	$$
	\frac{\int_{0}^{\infty} w^{-1/2} \exp({-\tilde{r}w/2}) p_{G(\alpha)}(w)dw}{\int_{0}^{\infty} w^{1/2} \exp({-\tilde{r}w/2}) p_{G(\alpha)}(w)dw}  = \frac{1}{2\alpha-1}\tilde{r} + \frac{2}{2\alpha-1}. 
	$$
	Therefore, if $\alpha>1/2$, $\eta_{p_{G(\alpha)}} = 1/(2\alpha-1)$.
	
	Let $\mixdensity^{\tilde{d}}(\cdot)$ be a mixing density proportional to $w^{(\tilde{d}-1)/2} \mixdensity (\cdot)$. Note that 
	$$		\frac{\int_{0}^{\infty} w^{(\tilde{d}-2)/2} \exp({-\tilde{r}w/2}) P_{\text{mix}}{(\df w)}}{\int_{0}^{\infty} w^{\tilde{d}/2} \exp({-\tilde{r}w/2}) P_{\text{mix}}{(\df w)}} = \frac{\int_{0}^{\infty} w^{-1/2} \exp({-\tilde{r}w/2}) \mixdensity^{\tilde{d}}(w)dw}{\int_{0}^{\infty} w^{1/2} \exp({-\tilde{r}w/2}) \mixdensity^{\tilde{d}}(w)dw}.$$
	Therefore, it suffices to prove that $\eta_{\mixdensity^{\tilde{d}}}< 1/(n-p+m)$.
	
	When $P_{\text{mix}}(\cdot)$ is polynomial near the origin with power $c>c_1$, 
	the distribution associated with $\mixdensity^{\tilde{d}}(\cdot)$ is polynomial near the origin with power $c+(\tilde{d}-1)/2$. Let $p^{\tilde{d}}_G(\cdot)$ be a mixing density following Gamma $(c+(\tilde{d}+1)/2, 1)$. We have $$\lim_{w\to 0}\frac{\mixdensity^{\tilde{d}}(w)}{p^{\tilde{d}}_G(w)}\in(0,\infty).$$ By Lemma 1 in \cite{hobert2018convergence}, for all $\tilde{d} \in [\min\{d_1, \dots, d_n\}, d]$, 
	$$\eta_{\mixdensity^{\tilde{d}}} = \eta_{p^{\tilde{d}}_G} = \frac{1}{2c+\tilde{d}} < \frac{1}{n-p+m}.$$

	When $P_{\text{mix}}(\cdot)$ is faster than polynomial near the origin, again let $\mixdensity^{\tilde{d}}(\cdot)$ be a mixing density proportional to $w^{(\tilde{d}-1)/2} \mixdensity (\cdot)$. 
	The distribution associated with $\mixdensity^{\tilde{d}}(\cdot)$ is faster than polynomial near the origin. Let $\alpha>1/2$. Recall that $p_{G(\alpha)}(w)$  is the density of the Gamma $(\alpha, 1)$ mixing distribution. By the definition of faster than polynomial near the origin, there exists $\kappa_{\alpha-1}>0$, such that $\mixdensity^{\tilde{d}}(w)/w^{\alpha-1}$ is strictly increasing on $(0, \kappa_{\alpha-1})$. Thus, $\mixdensity^{\tilde{d}}(w)/p_{G(\alpha)}(w)$ is strictly increasing on $(0, \kappa_{\alpha-1})$. By Lemma 2 in \cite{hobert2018convergence}, $$\eta_{\mixdensity^{\tilde{d}}} \leq \eta_{p_{G(\alpha)}} = \frac{1}{2\alpha-1}.$$
	Since $\alpha$ can be any value larger than $1/2$, $\eta_{\mixdensity^{\tilde{d}}} = 0$, for all  $\tilde{d} \in [\min\{d_1, \dots, d_n\}, d]$.

\end{proof}


\section{The ESS and ESSpm of Each Component in Section \ref{sec:num}} \label{app:ess}

\begin{table}[H]
	\centering
	\caption{
		The ESS of each component in different scenarios. The model column lists the mixing distributions and samplers used, and the data column labels the missing structures by the number of fully observed rows.}\label{app:exp2ess}
	\begin{tabular}{lllllllll}
		\hline
		Model & Data & $\bm{B}_{11}$ & $\bm{B}_{21}$ & $\bm{B}_{12}$ & $\bm{B}_{22}$ & $\bm{\Sigma}_{11}$ & $\bm{\Sigma}_{21}$ & $\bm{\Sigma}_{22}$ \\ \hline
		DAG & 45 & 19332 & 17450 & 19098 & 16447 & 10290 & 11646 & 10528 \\ 
		DAG & 40 & 20001 & 18399 & 17549 & 19992 & 10315 & 12061 & 10714 \\ 
		DAG & 35 & 16960 & 18459 & 16786 & 16967 & 11605 & 12442 & 11551 \\ 
		DAIG & 45 & 16471 & 19242 & 17976 & 19443 & 9970 & 10723 & 9888 \\ 
		DAIG & 40 & 14292 & 15528 & 18045 & 17066 & 8335 & 11155 & 9897 \\ 
		DAIG & 35 & 13330 & 17612 & 17745 & 16203 & 9676 & 8988 & 10064 \\ 
		DAGIG & 45 & 16843 & 17429 & 14810 & 15401 & 13796 & 13232 & 13422 \\ 
		DAGIG & 40 & 20661 & 20371 & 21888 & 14722 & 13375 & 14279 & 12765 \\ 
		DAGIG & 35 & 17729 & 15738 & 17347 & 15574 & 15619 & 12788 & 12722 \\ 
		DAIGIG & 45 & 18267 & 15193 & 17571 & 15148 & 11803 & 12759 & 13323 \\ 
		DAIGIG & 40 & 16540 & 18311 & 16763 & 17924 & 14042 & 14860 & 14749 \\ 
		DAIGIG & 35 & 12246 & 14077 & 18134 & 15717 & 8973 & 9506 & 12939 \\ 
		DAP & 45 & 30000 & 30000 & 30000 & 30000 & 30000 & 29724 & 30000 \\ 
		DAP & 40 & 30000 & 30000 & 30000 & 30000 & 30000 & 30000 & 30438 \\ 
		DAP & 35 & 30000 & 30000 & 29216 & 30000 & 30000 & 30000 & 30000 \\ 
		DAIP & 45 & 22822 & 26402 & 30000 & 30000 & 25435 & 29601 & 30000 \\ 
		DAIP & 40 & 20145 & 23817 & 30000 & 30000 & 20160 & 26305 & 30000 \\ 
		DAIP & 35 & 13894 & 22094 & 30000 & 30000 & 16530 & 13943 & 30000 \\   \hline
	\end{tabular}
\end{table}

\begin{table}[H]
	\centering
	\caption{The ESSpm of each component in different scenarios. The model column lists the mixing distributions and samplers used, and the data column labels the missing structures by the number of fully observed rows.}\label{app:exp2esspm}
	\begin{tabular}{lllllllll}
		\hline
		Model & Data & $\bm{B}_{11}$ & $\bm{B}_{21}$ & $\bm{B}_{12}$ & $\bm{B}_{22}$ & $\bm{\Sigma}_{11}$ & $\bm{\Sigma}_{21}$ & $\bm{\Sigma}_{22}$ \\ \hline
		DAG & 45 & 4421 & 3990 & 4367 & 3761 & 2353 & 2663 & 2407 \\ 
		DAG & 40 & 4886 & 4495 & 4287 & 4884 & 2520 & 2947 & 2618 \\ 
		DAG & 35 & 4213 & 4585 & 4170 & 4215 & 2883 & 3090 & 2869 \\ 
		DAIG & 45 & 3286 & 3838 & 3586 & 3878 & 1989 & 2139 & 1972 \\ 
		DAIG & 40 & 2891 & 3141 & 3651 & 3452 & 1686 & 2257 & 2002 \\ 
		DAIG & 35 & 2850 & 3766 & 3794 & 3464 & 2069 & 1922 & 2152 \\ 
		DAGIG & 45 & 1611 & 1667 & 1417 & 1473 & 1320 & 1266 & 1284 \\ 
		DAGIG & 40 & 2072 & 2043 & 2195 & 1477 & 1341 & 1432 & 1280 \\ 
		DAGIG & 35 & 1803 & 1601 & 1765 & 1584 & 1589 & 1301 & 1294 \\ 
		DAIGIG & 45 & 1701 & 1415 & 1636 & 1410 & 1099 & 1188 & 1240 \\ 
		DAIGIG & 40 & 1543 & 1709 & 1564 & 1672 & 1310 & 1387 & 1376 \\ 
		DAIGIG & 35 & 1201 & 1381 & 1779 & 1542 & 880 & 932 & 1269 \\ 
		DAP & 45 & 6959 & 6959 & 6959 & 6959 & 6959 & 6895 & 6959 \\ 
		DAP & 40 & 7793 & 7793 & 7793 & 7793 & 7793 & 7793 & 7907 \\ 
		DAP & 35 & 7478 & 7478 & 7282 & 7478 & 7478 & 7478 & 7478 \\ 
		DAIP & 45 & 4929 & 5702 & 6479 & 6479 & 5493 & 6393 & 6479 \\ 
		DAIP & 40 & 4415 & 5219 & 6574 & 6574 & 4418 & 5765 & 6574 \\ 
		DAIP & 35 & 3183 & 5062 & 6873 & 6873 & 3787 & 3194 & 6873 \\    \hline
	\end{tabular}
\end{table}

\begin{table}[H]
	\centering
	\caption{
		The ESS of each component in different scenarios. The model column lists the Gamma$(v,v)$ mixing distributions and samplers used, and the data column labels the missing structures by the number of fully observed rows.
	}\label{app:exp3ess}
	\begin{tabular}{lllllllll}
		\hline
		Model & Data & $\bm{B}_{11}$ & $\bm{B}_{21}$ & $\bm{B}_{12}$ & $\bm{B}_{22}$ & $\bm{\Sigma}_{11}$ & $\bm{\Sigma}_{21}$ & $\bm{\Sigma}_{22}$ \\ \hline
		DAG2 & 45 & 19332 & 17450 & 19098 & 16447 & 10290 & 11646 & 10528 \\ 
		DAG2 & 40 & 20001 & 18399 & 17549 & 19992 & 10315 & 12061 & 10714 \\ 
		DAG2 & 35 & 16960 & 18459 & 16786 & 16967 & 11605 & 12442 & 11551 \\ 
		DAIG2 & 45 & 16471 & 19242 & 17976 & 19443 & 9970 & 10723 & 9888 \\ 
		DAIG2 & 40 & 14292 & 15528 & 18045 & 17066 & 8335 & 11155 & 9897 \\ 
		DAIG2 & 35 & 13330 & 17612 & 17745 & 16203 & 9676 & 8988 & 10064 \\ 
		DAG8 & 45 & 26359 & 23824 & 25032 & 25504 & 16537 & 19326 & 16185 \\ 
		DAG8 & 40 & 26122 & 26707 & 25362 & 26588 & 14952 & 19118 & 15616 \\ 
		DAG8 & 35 & 25470 & 24347 & 24819 & 25501 & 18445 & 16390 & 16759 \\ 
		DAIG8 & 45 & 20631 & 23024 & 25552 & 26183 & 13685 & 15991 & 15078 \\ 
		DAIG8 & 40 & 18122 & 21684 & 23664 & 25824 & 13061 & 18082 & 16806 \\ 
		DAIG8 & 35 & 14837 & 22029 & 24967 & 26256 & 11553 & 11712 & 15336 \\ 
		DAG25 & 45 & 29041 & 29055 & 27919 & 28865 & 21494 & 22572 & 23345 \\ 
		DAG25 & 40 & 28018 & 29302 & 28548 & 29081 & 23303 & 21042 & 23013 \\ 
		DAG25 & 35 & 29310 & 29162 & 29052 & 30000 & 18884 & 16623 & 20118 \\ 
		DAIG25 & 45 & 23968 & 24454 & 27999 & 29108 & 17254 & 20737 & 22495 \\ 
		DAIG25 & 40 & 18613 & 23362 & 28295 & 28963 & 15049 & 18725 & 21078 \\ 
		DAIG25 & 35 & 17014 & 20338 & 28073 & 29150 & 10450 & 9397 & 22567 \\   \hline
	\end{tabular}
\end{table}

\begin{table}[H]
	\centering
	\caption{The ESSpm of each component in different scenarios. The model column lists the Gamma$(v,v)$ mixing distributions and samplers used, and the data column labels the missing structures by the number of fully observed rows.}\label{app:exp3esspm}
	\begin{tabular}{lllllllll}
		\hline
		Model & Data & $\bm{B}_{11}$ & $\bm{B}_{21}$ & $\bm{B}_{12}$ & $\bm{B}_{22}$ & $\bm{\Sigma}_{11}$ & $\bm{\Sigma}_{21}$ & $\bm{\Sigma}_{22}$ \\ \hline
		DAG2 & 45 & 4421 & 3990 & 4367 & 3761 & 2353 & 2663 & 2407 \\ 
		DAG2 & 40 & 4886 & 4495 & 4287 & 4884 & 2520 & 2947 & 2618 \\ 
		DAG2 & 35 & 4213 & 4585 & 4170 & 4215 & 2883 & 3090 & 2869 \\ 
		DAIG2 & 45 & 3286 & 3838 & 3586 & 3878 & 1989 & 2139 & 1972 \\ 
		DAIG2 & 40 & 2891 & 3141 & 3651 & 3452 & 1686 & 2257 & 2002 \\ 
		DAIG2 & 35 & 2850 & 3766 & 3794 & 3464 & 2069 & 1922 & 2152 \\ 
		DAG8 & 45 & 6129 & 5539 & 5820 & 5930 & 3845 & 4494 & 3763 \\ 
		DAG8 & 40 & 6495 & 6641 & 6306 & 6611 & 3718 & 4754 & 3883 \\ 
		DAG8 & 35 & 6632 & 6339 & 6462 & 6640 & 4803 & 4268 & 4364 \\ 
		DAIG8 & 45 & 4095 & 4570 & 5072 & 5197 & 2717 & 3174 & 2993 \\ 
		DAIG8 & 40 & 3894 & 4659 & 5085 & 5549 & 2806 & 3885 & 3611 \\ 
		DAIG8 & 35 & 3265 & 4847 & 5494 & 5777 & 2542 & 2577 & 3375 \\ 
		DAG25 & 45 & 7115 & 7118 & 6840 & 7072 & 5266 & 5530 & 5719 \\ 
		DAG25 & 40 & 7097 & 7423 & 7232 & 7367 & 5903 & 5330 & 5830 \\ 
		DAG25 & 35 & 7728 & 7689 & 7660 & 7910 & 4979 & 4383 & 5304 \\ 
		DAIG25 & 45 & 5067 & 5170 & 5919 & 6154 & 3648 & 4384 & 4755 \\ 
		DAIG25 & 40 & 4099 & 5144 & 6231 & 6378 & 3314 & 4123 & 4642 \\ 
		DAIG25 & 35 & 3898 & 4660 & 6432 & 6679 & 2394 & 2153 & 5170 \\   \hline
	\end{tabular}
\end{table}




\newpage

\renewcommand{\thepage}{}
\bibliographystyle{plainnat}
\bibliography{JMVA_incomplete}
\end{document}